\documentclass[11pt]{amsart}
\usepackage[table]{xcolor}
\usepackage{amssymb,graphicx,tikz}
\usepackage{amsthm}
\usepackage{amsmath}
\newcommand{\TryPackage}[3]{\IfFileExists{#1.sty}{\usepackage{#1}#2}{#3}}
\TryPackage{mathrsfs}{\renewcommand{\mathcal}{\mathscr}}{%
     \TryPackage{eucal}{}{}}
\usepackage[colorlinks=true,urlcolor=blue,linkcolor=blue,citecolor=blue]{hyperref}
\usepackage[noabbrev]{cleveref}

\usepackage[normalem]{ulem} 
\usepackage{soul} 
\setstcolor{red}
\usepackage{caption}
\usepackage{subcaption}

\newtheorem{theorem}{Theorem}

\newtheorem{conjecture}[theorem]{Conjecture}
\newtheorem{corollary}[theorem]{Corollary}
\newtheorem{definition}[theorem]{Definition}

\newtheorem{lemma}[theorem]{Lemma}
\newtheorem{proposition}[theorem]{Proposition}

\theoremstyle{remark}
\newtheorem{remark}[theorem]{Remark}
\newtheorem{example}[theorem]{Example}

\newcommand{\bpr}{\begin{proof}}
\newcommand{\epr}{\end{proof}}

\newcommand{\la}{\langle}
\newcommand{\ra}{\rangle}
\newcommand{\comment}[1]{\,}
\newcommand{\wh}{\widehat}
\newcommand{\cal}{\mathcal}
\newcommand{\p}{\partial}

\newcommand{\Z}{\mathbb Z}

\newcommand{\R}{\mathbb R}
\newcommand{\C}{\mathbb C}

\newcommand{\cS}{\mathfrak S}

\newcommand{\sL}{\mathscr L}
\newcommand{\sS}{\mathscr S}

\newcommand{\sm}{\smallsetminus}
 \newcommand{\lto}{\longrightarrow}

\newcommand{\diag}[2]{\parbox{#2}{\includegraphics[width=#2]{Figures/#1.pdf}}}

\newcommand{\KP}[1]{%
  \begin{tikzpicture}[baseline=-\dimexpr\fontdimen22\textfont2\relax]
  #1
  \end{tikzpicture}%
}
\newcommand{\KPC}{%
  \KP{\filldraw[color=black, fill=none, thick] circle (0.18);}%
}
\newcommand{\KPX}{%
  \KP{
    \draw[color=black, thick] (-0.3,-0.3) -- (0.3,0.3);
    \draw[color=black, thick] (-0.3,0.3) -- (-0.05,0.05);
    \draw[color=black, thick] (0.05,-0.05) -- (0.3,-0.3);
  }%
}
\newcommand{\KPB}{%
  \KP{%
    \draw[color=black, thick] (-0.3,0.3) .. controls (0,-0.02) .. (0.3,0.3);
    \draw[color=black, thick] (-0.3,-0.3) .. controls (0,0.02) .. (0.3,-0.3);
  }%
}
\newcommand{\KPA}{%
  \KP{%
    \draw[color=black, thick] (-0.3,-0.3) .. controls (0.02,0) .. (-0.3,0.3);
    \draw[color=black, thick] (0.3,-0.3) .. controls (-0.02,0) .. (0.3,0.3);
  }%
}
\newcommand{\KPXP}{%
  \KP{
    \draw[color=black, thick,-stealth] (-0.3,0.3) -- (0.3,-0.3);
    \draw[color=black, thick] (-0.3,-0.3) -- (-0.05,-0.05);
    \draw[color=black, thick,-stealth] (0.05,0.05) -- (0.3,0.3);
  }%
}
\newcommand{\KPXN}{%
  \KP{
    \draw[color=black, thick,-stealth] (-0.3,-0.3) -- (0.3,0.3);
    \draw[color=black, thick] (-0.3,0.3) -- (-0.05,0.05);
    \draw[color=black, thick,-stealth] (0.05,-0.05) -- (0.3,-0.3);
  }%
}
\newcommand{\KPXO}{%
  \KP{
    \draw[color=black, thick] (-0.4,-0.1)  .. controls (-0.25,-0.2) and (-0.15,0.1) .. (0,0.10).. controls (0.15,0.1) and (0.25,-0.2)..  (0.4,-0.1);
  }%
}
\newcommand{\KPXQ}{%
  \KP{%
    \draw[color=black, thick] (-0.45,-0.2) .. controls (0.52,0.3) and (-0.7,0.3) .. (-0.25,0.00);
    \draw[color=black, thick] (0.45,-0.2) .. controls (-0.52,0.3) and (0.7,0.3) .. (0.25,0.00);
    \draw[color=black, thick] (-0.14,-0.08) .. controls (-0.04,-0.15) and (0.04,-0.15) .. (0.14,-0.08);
  }%
}

\newcommand{\KPXR}{%
  \KP{
    \draw[color=black, thick] (-0.45,-0.2) -- (-0.3,-0.1);
    \draw[color=black, thick] (0.45,-0.2) -- (0.3,-0.1);
    \draw[color=black, thick] (-0.2,-0.03) .. controls (0.25,0.3) and (-0.8,0.3) ..  (-0.18,-0.12) .. controls (-0.04,-0.22) and (0.04,-0.22) .. (0.18,-0.12) .. controls (0.8,0.3) and (-0.25,0.3) ..(0.2,-0.03);
  }%
}

\title[Adequate links and the generalized Tait conjectures]{Adequate links in thickened surfaces and \\the generalized Tait conjectures}

\author{Hans U. Boden, Homayun Karimi, Adam S. Sikora}

\thanks{The first author was partially funded by the Natural Sciences and Engineering Research Council of Canada.}

\begin{document}

\thispagestyle{empty}

\begin{abstract}
In this paper, we apply Kauffman bracket skein algebras to develop a theory of skein adequate links in thickened surfaces.
We show that any alternating link diagram on a surface is skein adequate. We apply our theory to establish the first and second Tait conjectures for adequate links in thickened surfaces.
Our notion of skein adequacy is broader and more powerful than the corresponding notions of adequacy previously considered for link diagrams in surfaces.

For a link diagram $D$ on a surface $\Sigma$ of minimal genus $g(\Sigma)$, we show that
$${\rm span}([D]_\Sigma) \leq 4c(D) + 4 |D|-4g(\Sigma),$$
where $[D]_\Sigma$ is its skein bracket, $|D|$ is the number of connected components of $D$, and $c(D)$ is the number of crossings.
This extends a classical result of Kauffman, Murasugi, and Thistlethwaite. We further show that the above inequality is an equality if and only if $D$ is weakly alternating.
This is a generalization of a well-known result for classical links due to Thistlethwaite.
Thus the skein bracket  detects the crossing number for weakly alternating links. As an application, we show that the crossing number is additive under connected sum for adequate links in thickened surfaces.
\end{abstract}

\address{Mathematics \& Statistics, McMaster University, Hamilton, Ontario}
\email{boden@mcmaster.ca}
\address{Mathematics \& Statistics, McMaster University, Hamilton, Ontario}
\email{karimih@math.mcmaster.ca}
\address{Dept. of Mathematics, University at Buffalo, SUNY, Buffalo, NY 14260}
\email{asikora@buffalo.edu}


\pagestyle{myheadings}

\maketitle

\section{Introduction}

The Kauffman bracket is a $\Z[A^{\pm 1}]$-valued invariant of framed links in $\R^3$ determined by the skein relations:

\begin{equation}\label{e-KB}
\KPX-A \KPA -A^{-1}\KPB\ \text{and}\ \KPC-\delta,
\end{equation}
where $\delta=-A^2-A^{-2}.$

It naturally extends to an invariant of framed links in an arbitrary oriented $3$-manifold $M$ (possibly with boundary), via the skein module construction:
let $\sL(M)$ be the set of all unoriented, framed links in $M,$ including the empty link $\varnothing.$ The {\bf skein module} $\sS(M)$ of $M$ is the quotient of the free $\Z[A^{\pm 1}]$-module spanned by $\sL(M)$ by the submodule generated by the Kauffman bracket skein relations \eqref{e-KB}, cf. \cite{Przytycki-1999}, \cite{Turaev-1990, Turaev-1991}.

By this construction, the bracket
$$[\, \cdot \,] \colon  \sL(M)\to \sS(M),$$
sending framed links to their equivalence classes in $\sS(M)$, called the {\bf skein bracket}, is the universal invariant of framed links in $M$ satisfying  \eqref{e-KB}.

Independently of this initial motivation, skein modules quickly began to play a  much broader role in the development of quantum topology, for example in connection with $SL(2,\C)$ character varieties \cite{Bullock-1997, PS-2000, FKL-2019, Turaev-1991, BFK-1999},  topological quantum field theory, \cite{BHMV-1995, Turaev-1994}, (quantum) Teichm\"uller spaces and (quantum) cluster algebras \cite{BW-2011, CL-2019, FGo-2006, FST-2008, Muller-2016},  the AJ conjecture \cite{FGL-2002, Le-2006}, and many more.

In this paper we develop a general theory of skein adequacy (called adequacy, for short) for links in thickened surfaces with the aid of skein modules.

Let $\Sigma$ be an oriented surface and $I=[0,1]$ be the unit interval. The skein module of the thickened surface $\Sigma \times I$ comes naturally equipped with a product structure given by stacking, i.e., the product $L_1\cdot L_2$ is defined by placing $L_1$ on top of $L_2$ in $\Sigma \times I$. With this product structure, the skein module $\sS(\Sigma \times I)$ becomes an algebra over $\Z[A^{\pm 1}]$.

Let $\cal{C}(\Sigma)$  denote the set of all non-trivial unoriented simple loops on $\Sigma$ up to isotopy and $\cal{MC}(\Sigma)$ denote the set of all non-trivial unoriented multi-loops on $\Sigma$, i.e., collections of pairwise disjoint simple non-contractible loops, including $\varnothing$, up to isotopy. Then by \cite{Przytycki-1999} (cf., \cite{Sikora-Westbury}), the skein module $\sS(\Sigma\times I)$ is a free $\Z[A^{\pm 1}]$-module with basis $\cal{MC}(\Sigma)$. Consequently, via this identification, the skein bracket gives a map
\begin{equation}\label{e-bra}
[\, \cdot \, ]_\Sigma \colon \sL(\Sigma\times I)\to \sS(\Sigma\times I)=\Z[A^{\pm 1}]\cal{MC}(\Sigma).
\end{equation}

We use the association \eqref{e-bra} to develop a theory of skein adequacy for links in $\Sigma\times I$ which extends that for classical links. This theory is broader and more powerful than the corresponding notions of simple adequacy \cite{Lickorish-Thistlethwaite} and homological adequacy  \cite{Boden-Karimi-2019}. For example, we will see that every weakly alternating link in $\Sigma\times I$ without removable nugatory crossings is skein adequate.

We will apply the skein bracket to establish the first and the second Tait conjecture for skein adequate link diagrams on surfaces.  The first one says that skein adequate diagrams have minimal crossing number, and the second one says that two skein adequate diagrams for the same oriented link have the same writhe. (The writhe of a link diagram $D$ is denoted by $w(D)$ and is defined to be the sum of its crossing signs.)
These results strengthen the earlier work of Adams et al., \cite{Adams}, who showed the minimal crossing number result for reduced alternating knot diagrams in surfaces. We also strengthen the minimality result of \cite{Boden-Karimi-2019}, for homologically adequate link diagrams in surfaces, and further show that any connected sum of two skein adequate link diagrams on surfaces is again skein adequate. This implies that the crossing number and writhe are essentially additive under connected sum of
skein adequate links in thickened surfaces.

For any link diagram $D$ on a surface $\Sigma$ of minimal genus, we prove that
$${\rm span}([D]_\Sigma) \leq 4c(D) + 4 |D|-4g(\Sigma),$$
where $|D|$ is the number of connected components of $D$, $c(D)$ is the number of crossings, and $g(\Sigma)$ is the genus of $\Sigma.$
This  inequality generalizes a result proved by Kauffman, Murasugi, and Thistlethwaite for link diagrams on $\R^2$ \cite{Kauffman-87, Murasugi-871, Thistlethwaite-87},  extending their
nice geometric application of the Kauffman  bracket.
 It also extends and strengthens an analogous recent result proved in \cite{Boden-Karimi-2019} using the homological Kauffman bracket.

Additionally, we prove that the above inequality is an equality if and only if $D$ is weakly alternating. Therefore, the skein bracket, together with the crossing number, distinguishes weakly alternating links. That generalizes the analogous result of Thistlethwaite for classical links.

\subsection*{Broader context and motivation}
While the results presented here are new only for links in non-contractible surfaces, generalized link theory is of growing interest and has many potential connections to classical links and 3-dimensional geometry. We take a moment to discuss some of them.

One motivation for our results is their connection to the theory of virtual knots and links, which can be viewed as links in thickened surfaces, considered up to homeomorphisms and stabilization \cite{Carter-Kamada-Saito}. By Kuperberg's theorem, minimal genus realizations of virtual links are unique up to homeomorphism \cite{Kuperberg}. Our theory of adequate and alternating links in thickened surfaces is invariant under surface homeomorphisms and, therefore, many of the results given here can be restated in the language of virtual links.

A second motivation involves potentially novel applications to classical link theory. The Turaev surface construction associates to any classical link diagram  an alternating link in a thickened surface \cite{Turaev-1987, DFKLS, CK-Turaev}.
Menasco famously proved hyperbolicity  for prime alternating (non-torus) links in $S^3$ \cite{Menasco-1984}, and his result has been extended to prime alternating links
$L \subset \Sigma \times I$ in \cite{Adams-2019a}.
This result opens the door to using the hyperbolic geometry of alternating links in higher genus surfaces to profitably study non-alternating classical links, e.g., see \cite{Adams-2019c} and the many other papers cited below.

In \cite{Dasbach-Lin-2007},
Dasbach and Lin proved a remarkable result giving a bound on the volume of alternating link complements in terms of the second and penultimate coefficients of the Jones polynomial.
In \cite{Lackenby-2004}, Lackenby established an equally remarkable bound on the volume of alternating link complements in terms of the diagrammatic \emph{twist number}
For alternating hyperbolic links in $S^3$, the results of \cite{Dasbach-Lin-2007} imply that the twist number is essentially an isotopy invariant of $L$, but this is not true in general.

These methods have been generalized to non-alternating hyperbolic links in $S^3$ \cite{Blair-2009, Blair-Allen-Rodriguez-2019}
and to hyperbolic links in arbitrary compact oriented 3-manifolds \cite{HP-2020}.
In general, there is a notion of weakly generalized alternating link diagrams on surfaces due to Howie \cite{Howie-2015}, extended to links in compact oriented 3-manifolds via ``generalized projection surfaces'' by Howie and Purcell \cite{HP-2020}.

The volume bounds have been extended to alternating links in thickened surfaces by Bavier and Kalfagianni \cite{Bavier-Kalfagianni-2020} and Will \cite{Will-2020} and also to virtual alternating links by Champanerkar and Kofman \cite{Champanerkar-Kofman-2020}.
In \cite{Champanerkar-Kofman-2020} and \cite{Will-2020}, the volume bounds are expressed in terms of the Jones-Krushkal polynomial \cite{Krushkal-2011, Boden-Karimi-2019}, and  in \cite{Bavier-Kalfagianni-2020}
they are expressed in terms of a skein invariant derived from fully contractible smoothings.
In \cite[Corollary 1.3]{Bavier-Kalfagianni-2020}, they deduce that, for certain alternating links in thickened surfaces, the twist number is an isotopy invariant. Interestingly, this result is consistent with the generalized Tait flyping conjecture.

\section{State sum formula and the generalized Jones polynomial}
\label{s-state-sum}

We will assume throughout this paper that $\Sigma$ is an oriented surface with one or more connected components, which may also have boundary. Links in $\Sigma\times I$  will be represented as diagrams on $\Sigma$ up to Reidemeister moves.

Every framed link in $\Sigma\times I$  can also be represented by a link diagram with framing given by the blackboard framing. Equivalence of framed links is given by regular isotopy, which includes the second and third Reidemeister moves, as well as the modified first Reidemeister move, which replaces $\KPXR$ or $\KPXQ$ with $\ \KPXO \ .$

Let $D$ be a link diagram on a surface $\Sigma$. Given a crossing \KPX\ of $D$, we consider its $A$-type \KPA and $B$-type \KPB resolution, as in the Kauffman bracket construction. A choice of resolution for each crossing of $D$ is called a {\bf state}. Let $\cS(D)$ denote the set of all states of $D$. Thus $|\cS(D)|=2^{c(D)},$ where $c(D)$ is the crossing number of $D$.

For $S\in \cS(D)$, let $|S|$ denote the number of loops in $S$ and $t(S)$ the number of contractible loops in $S$. Also let $\wh S$ denote $S$ with contractible loops removed. Hence, $\wh S\in \cal{MC}(\Sigma).$ The following state sum formula  is an immediate consequence of the definition, and it generalizes the usual formula for the classical Kauffman bracket:
\begin{equation}\label{e-state-sum}
[D]_\Sigma=\sum_{S\in \cS(D)}  A^{a(S)-b(S)}\delta^{t(S)}\wh S \in  \Z[A^{\pm 1}]\cal{MC}(\Sigma),
\end{equation}
where $a(S),b(S)$ are the numbers of $A$- and $B$-smoothings in $S$ and $\delta = -A^2-A^{-2}$ as before. A similar formula appears in the paper of Dye and Kauffman on the surface bracket polynomial  \cite{Dye-Kauffman-2005}.

Any invariant of framed links in $\Sigma\times I$ satisfying \eqref{e-KB} can be normalized to obtain a Jones-type polynomial invariant of oriented links. In the case of the skein bracket \eqref{e-bra}, one obtains the {\bf generalized Jones polynomial}, an invariant for oriented links in $\Sigma \times I$ given by
\begin{equation}\label{e-Jones}
J_{\Sigma}(D)=(-1)^{w(D)}t^{3w(D)/4}([D]_\Sigma)_{A=t^{-1/4}}.
\end{equation}


\section{Adequate link diagrams in surfaces}
\label{s-adeq}

Given a link diagram $D$, let $S_A$ be the pure $A$ state and let $S_B$ be the pure $B$ state. Then $S_A$ and $S_B$ are the states which theoretically give rise to the terms of maximal and minimal degree in \eqref{e-state-sum}. The notion of adequacy of a link diagram is designed to guarantee that the terms from $S_A$ and $S_B$ survive in the state sum formula. Therefore, when $D$ is a skein adequate diagram, its skein bracket $[D]_\Sigma$ has maximal possible span.

Two states $S, S'$ are said to be {\bf adjacent} if their resolutions differ at exactly one crossing.

\begin{definition} \label{defn:h-adequate}
A link diagram $D$ on a surface $\Sigma$ is said to be $A$-adequate if $t(S) \leq t(S_A )$ or $\wh{S} \neq  \wh{S}_A$ in $\cal{MC}(\Sigma)$ for any state $S$ adjacent to $S_A$. It is  said to be $B$-adequate if $t(S) \leq t(S_B)$ or $\wh{S} \neq  \wh{S}_B$ for any state $S$ adjacent to $S_B$. The diagram $D$ is called skein adequate if it is both $A$- and $B$-adequate.
\end{definition}

The notions of $A$- and $B$-adequacy are modeled on the notions of plus- and minus-adequacy for classical links \cite{Lickorish}. Recall that a classical link diagram is said to be {\bf plus-adequate} if $|S| = |S_A|-1$ for any state $S$ adjacent to $S_A$, and it is {\bf minus-adequate} if $|S| = |S_B|-1$ for any state $S$ adjacent to $S_B$. This simpler notion of adequacy extends verbatim to link diagrams on surfaces. For link diagrams on surfaces, plus- and minus-adequacy is a special case of the notion of homological adequacy, which was introduced in \cite{Boden-Karimi-2019} and will be reviewed in 
\Cref{s-homological}. We will see that adequacy as defined above is more general than simple or homological adequacy.

The following provides an alternative definition of adequacy:

\begin{proposition} \label{p-alt-adeq}
(1) A link diagram $D$ on $\Sigma$ is $A$-adequate if and only if $t(S) \leq t(S_A )$ or $|\wh{S}|\neq |\wh{S}_A|$ for any state $S$ adjacent to $S_A$.\\
(2) A link diagram $D$ on $\Sigma$ is $B$-adequate if and only if $t(S) \leq t(S_B)$ or $|\wh{S}| \neq  |\wh{S}_B|$ for any state $S$ adjacent to $S_B$.
\end{proposition}

\begin{proof}
We begin with some general comments. Given a link diagram $D$ and two adjacent states $S,S'$, the transition from $S$ to $S'$ is one of the following  types:
\begin{itemize}
\item[(i)] $|S'| = |S|+1$, i.e., one cycle of $S$ splits into two cycles of $S'$.
\item[(ii)] $|S'| = |S|-1$, i.e., two cycles of $S$ merge into one cycle of $S'$.
\item[(iii)] $|S'| = |S|,$ i.e., one cycle $C$ of $S$ rearranges itself into a new cycle $C'$ of $S'$.\footnote{The transition $S \to S'$ in this case is called a \textit{single cycle bifurcation.}}
\end{itemize}

In cases (ii) and (iii), either $t(S')\leq t(S)$ or  $\wh{S'} \neq  \wh{S}$.
Specifically, in case (ii),  $t(S')  > t(S)$ only when two non-trivial parallel cycles in $S$ merge to form one trivial cycle in $S'$, which implies that $\wh{S} \neq \wh{S'}$.
Likewise, in case (iii), we claim that neither $C$ nor $C'$ is trivial and, consequently, $t(S')  = t(S)$. To see that, note that if $S'$ is obtained from $S$ by a smoothing change of a crossing $x$ then
there are two simple closed loops $\alpha,\beta\subset \Sigma$ intersecting at $x$ only and such that the two different smoothings of $x$ yield $C$ and $C'$. Assigning some orientations to $\alpha$ and $\beta$, we see that $C$ and $C'$ with some orientations equal $\pm (\alpha+\beta)$ and $\pm (\alpha-\beta)$ in $H_1(\Sigma).$ Since the algebraic intersection number of $\alpha$ and $\beta$ is $1$, we know that $\alpha\ne \pm \beta$ and, consequently, neither $C$ nor $C'$ is trivial.

Therefore, to verify that a given diagram is $A$- or $B$-adequate, it is enough to check that the conditions of \Cref{defn:h-adequate} hold in case (i).

We will now prove part (1). Suppose $S$ is a state adjacent to $S_A$ with $t(S)=t(S_A)+1$. Then the transition from $S_A$ to $S$ must either be case (i) or (ii).

If it is case (i), then  $|S|=|S_A|+1$ and $t(S)=t(S_A)+1$, therefore, $\wh{S} =\wh{S}_A.$ Thus $D$ is not $A$-adequate and $|\wh{S}|=|\wh{S}_A|.$ If it is case (ii), then $|S|=|S_A|-1$, and two nontrivial cycles of $S_A$ must merge into a trivial cycle of $S$. In this case, the conditions for $A$-adequacy are satisfied and $|\wh{S}|\neq|\wh{S}_A|.$

The proof of part (2) is similar and is left to the reader.
\end{proof}

For any diagram $D$, its bracket has a unique presentation
$$[D]_\Sigma=\sum_\mu p_\mu(D)\mu \in \sS(\Sigma\times I),$$
where the sum is over all multi-loops $\mu$ in $\Sigma.$ Denote the maximal and minimal degrees (in the variable $A$) of the non-zero polynomials $p_\mu(D)$ in this expression by $d_{max}([D]_\Sigma)$ and $d_{min}([D]_\Sigma).$

\begin{proposition}\label{p-dminmax}
For any link diagram $D$ on $\Sigma$,\\
(1) $d_{max}([D]_\Sigma) \leq c(D) + 2t(S_A)$, with equality if $D$ is $A$-adequate. \\
(2) $d_{min}([D]_\Sigma) \geq -c(D) -2t(S_B),$ with equality if $D$ is $B$-adequate.
\end{proposition}

\noindent{\it Proof  of (1).} By \eqref{e-state-sum}, $[D]_\Sigma$ is given by a state sum with the term $(-1)^{t(S_A)} A^{c(D)+2t(S_A)}\wh{S}_A$ for the state $S_A.$ Now the inequality of (1) follows from the fact that every change of a smoothing in $S_A$ decreases $a(S)-b(S)$ by two and increases $t(S)$ by at most one.

The proof of equality in (1) when $D$ is $A$-adequate follows immediately from part (1) of the lemma below.

The proof of (2) is analogous, and the proof of equality in (2) when $D$ is $B$-adequate follows from part (2) of the lemma below.
\qed

\begin{lemma} \label{lemma-hom-ad}
(1) If $D$ is $A$-adequate and $S$ is a state with at least one $B$-smoothing, then either
$$ a(S)-b(S)+2t(S)< c(D) + 2t(S_A)\ \quad \text{or} \quad \wh{S}\ne \wh{S}_A.$$
(2) If $D$ is $B$-adequate and $S$ is a state with at least one $A$-smoothing, then either
$$a(S)-b(S)+2t(S) > -c(D) - 2t(S_B)\ \quad \text{or} \quad \wh{S}\ne \wh{S}_B.$$
\end{lemma}

\begin{proof} We prove (1) by contradiction:
Suppose to the contrary that $S$ is a state with at least one $B$-smoothing
such that $\wh{S}=\wh{S}_A$ and
$$a(S)-b(S)+2t(S) = c(D) + 2t(S_A).$$
Clearly, $S$ can be obtained from $S_A$ by a sequence of smoothing changes from $A$ to $B$, $S_A=S_0\to S_1\to \cdots \to S_k=S$. Further, each smoothing change must increase $t(\cdot )$ by one, i.e., $t(S_{i+1})=t(S_i)+1$, for $i=0,\ldots,k-1.$ Since each smoothing change increases the number of cycles in a state by at most one, none of these smoothing changes can add a new cycle to $\wh S_i$, $i=0,\ldots,k$. Therefore, $|\wh{S}_{i+1}| \leq |\wh{S}_{i}|$ for $i=0,\ldots,k-1.$
However, since $\wh{S}=\wh{S}_A$, none of the smoothing changes can decrease $|\wh{S}_i|$ either. It follows that $\wh{S}_{i+1}=\wh{S}_i$  for $i=0,\ldots,k-1$. Thus
$|\wh{S}_{i+1}|=|\wh{S}_i|$ and
$$|S_{i+1}| =t(S_{i+1}) + |\wh{S}_{i+1}| = t(S_{i})+1 + |\wh{S}_{i}|= |S_{i}|+1,$$ for $i=0,\ldots,k-1$.
In particular, each transition  $S_i \to S_{i+1}$ is of type (i) as discussed in the proof of \Cref{p-alt-adeq}, i.e., one where a cycle of $S_i$
splits into two cycles of $S_{i+1}$.

However, since $D$ is $A$-adequate, the first smoothing change $S_A=S_0\to S_1$ has either $t(S_1)\leq t(S_A)$ or $\wh{S}_1 \neq \wh{S}_A$, which is a contradiction.

This completes the proof of the first statement. The proof of the second one is similar and is left to the reader.
\end{proof}

The next result is an immediate consequence of \Cref{p-dminmax}.
Below, ${\rm span}([D]_\Sigma)$ denotes the difference between the maximal and minimal $A$-degree of $[D]_\Sigma$.
\begin{corollary}\label{cor-adequate}
If $D$ is a link diagram on $\Sigma$, then
$${\rm span}([D]_\Sigma) \leq 2 c(D) + 2t(S_A)+2t(S_B),$$
with equality if $D$ is skein adequate.
\end{corollary}

The map $\Psi \colon \cal{MC}(\Sigma)\to \Z[z]$ sending $S$ to $z^{|S|}$ extends linearly to the skein module,
$$\Psi \colon \sS(\Sigma\times I)=\Z[A^{\pm 1}]\cal{MC}(\Sigma)\lto \Z[A^{\pm 1},z].$$
The composition $\Psi([D]_\Sigma)$ is called the {\bf reduced homotopy Kauffman bracket}.
Obviously,
$${\rm span}(\Psi([D]_\Sigma))\leq {\rm span}([D]_\Sigma),$$
where ${\rm span}(\ \cdot \ ) $ refers to the span in the $A$-degree.

\begin{proposition}
If $D$ is a skein adequate link diagram on $\Sigma,$ then
$${\rm span}(\Psi([D]_\Sigma))= {\rm span}([D]_\Sigma).$$
\end{proposition}

\begin{proof}
Let $S$ be a state with at least one $B$-smoothing such that $|\wh{S}|=|\wh{S}_A|$ and $a(S)-b(S)+2t(S)= c(D) + 2t(S_A)$. As before, $S$ can be obtained from $S_A$ by a sequence of smoothing changes from $A$ to $B$, each smoothing change can increase $t(\cdot)$ by at most one, i.e., $S_A=S_0\to S_1\to \cdots \to S_k=S$. As in the proof of \Cref{lemma-hom-ad}, we must have $t(S_{i+1})=t(S_i)+1$. Further, since a smoothing change can increase the number of cycles in $S_i$ by at most one, we have $|\wh{S}_{i+1}|\leq |\wh{S}_{i}|$ for $i=0,\ldots,k-1.$ Now the assumption that $|\wh{S}|=|\wh{S}_A|$ then implies that
$|\wh{S}_{i+1}| = |\wh{S}_{i}|$ for $i=0,\ldots,k-1.$
However, since $D$ is adequate, for the first transition $S_A=S_0 \to S_1$, either $t(S_{1}) \neq t(S_0)+1$ or $\wh{S}_1 \neq \wh{S}_0$. But $t(S_{1}) = t(S_0)+1$ and  $|\wh{S}_1| = |\wh{S}_0|$ imply that
$\wh{S}_1 = \wh{S}_0$, which gives a contradiction.

Therefore, the term with maximum $A$-degree in $\Psi([D]_{\Sigma})$ must survive. A similar argument applies to show that the term with minimum $A$-degree survives. It follows that
$${\rm span}(\Psi([D]_\Sigma))= 2c(D)+2t(S_A)+2t(S_B)={\rm span}([D]_\Sigma).$$
\end{proof}

The next proposition shows that skein adequacy is inherited under passing to subsurfaces $\Sigma' \subset \Sigma$.

\begin{proposition} If a link diagram $D$ on a subsurface $\Sigma'$ of $\Sigma$ is $A$- or $B$-adequate in $\Sigma$ then it is $A$- or $B$-adequate (respectively) in $\Sigma'$.
\end{proposition}

\begin{proof} In the following, let $t(S,\Sigma)$ be the value of $t(S)$ when $S$ is regarded as a state in $\Sigma$, and let $t(S,\Sigma')$ be its value when $S$ is regarded as a state in $\Sigma'$.

Suppose $D$ is not $A$-adequate in $\Sigma'$.
By \Cref{p-alt-adeq}, there exists a state $S$ adjacent to $S_A$ with $t(S,\Sigma')=t(S_A, \Sigma')+1$ and $|\wh{S}|=|\wh{S}_A|$ in $\Sigma'$. In particular, $|S|=|S_A|+1$, and the transition from $S_A$ to $S$ must involve one cycle $C$ of $S_A$ splitting into two cycles $C_1$ and $C_2$ of $S$.
At least one of $C_1, C_2$ must be trivial in $\Sigma'$, for otherwise $t(S,\Sigma') \leq t(S_A,\Sigma')$.
If say $C_1$ is trivial in $\Sigma'$, then it must also be trivial in $\Sigma,$
because $\Sigma' \subset \Sigma$ is a subsurface.

As a cycle in $\Sigma$, $C$ is either trivial or nontrivial.
If it is trivial, then $C_2$ must also be trivial in $\Sigma$, and so in fact all three of $C,C_1,C_2$ are trivial. This implies that $t(S, \Sigma) = t(S_A,\Sigma)+1$ and $|\wh{S}|=|\wh{S}_A|$ in $\Sigma,$ contradicting the assumption of $A$-adequacy of $D$.

If, on the other hand, $C$ is  nontrivial in $\Sigma,$ then $C_2$ must also be nontrivial in $\Sigma$. This again implies that $t(S, \Sigma) = t(S_A,\Sigma)+1$ and $|\wh{S}|=|\wh{S}_A|$ in $\Sigma,$ leading to the same contradiction.
Therefore, $D$ must be $A$-adequate on $\Sigma'$.

The proof of $B$-adequacy of $D$ is identical.
\end{proof}


\section{Skein and homological adequacy}
\label{s-homological}

For completeness of discussion, in this section we compare \Cref{defn:h-adequate} of skein adequacy to two legacy versions, namely simple and homological adequacy.
We will see that our notion of adequacy is broader and that the statements of \Cref{lemma-hom-ad} and \Cref{cor-adequate} are strictly stronger than the corresponding statements for simple and homological adequacy.
Henceforth, we will say a link diagram on a surface is adequate if it is skein adequate.

For any state $S\subset \Sigma,$ let us denote the ranks of the kernel and the image of
$$i_* \colon H_1(S;\Z/2)\to H_1(\Sigma;\Z/2),$$
by $k(S)$ and $r(S),$ respectively.

The {\bf homological Kauffman} {\bf bracket},
$$\la D\ra_\Sigma=\sum_{S\in \cS(D)} A^{a(S)-b(S)}\delta^{k(S)}z^{r(S)},$$
was introduced by Krushkal \cite{Krushkal-2011} and studied in \cite{Boden-Karimi-2019}.

Based on this invariant, \cite{Boden-Karimi-2019} introduced the notion of homological adequacy  for link diagrams in surfaces. A diagram $D$ on $\Sigma$ is {\bf homologically $A$-adequate} if $k(S) \leq k(S_A) $ for any state $S$ adjacent to $S_A$, and it is {\bf homologically $B$-adequate} if $k(S) \leq k(S_B) $ for any state $S$ adjacent to $S_B$. A diagram $D$ is {\bf homologically adequate} if it is both homologically $A$- and $B$-adequate.

It is not difficult to show that a diagram that is plus-adequate is homologically $A$-adequate, and one that is minus-adequate is homologically $B$-adequate. (For further details, see \S 2.2 of \cite{Boden-Karimi-2019}.)

\begin{proposition}\label{p-class-homol-homot}
Every homologically $A$-adequate link diagram is $A$-adequate and every homologically $B$-adequate link diagram is $B$-adequate.
\end{proposition}

\begin{proof}
Recall from the discussion at the beginning of the proof of \Cref{p-alt-adeq} that there are the three possible cases, and  to verify that a given diagram is $A$- or $B$-adequate, it is enough to check that the conditions of \Cref{defn:h-adequate} hold in case (i).
Hence, it is enough to focus on states $S$ adjacent to $S_A$ or $S_B$ with $|S| = |S_A|+1$ or $|S| = |S_B|+1$, respectively.

If $D$ is not $A$-adequate, then there exists a state $S$ adjacent to $S_A$ with $|S| = |S_A|+1$, $t(S)=t(S_A)+1$,  and $\wh{S} = \wh{S}_A$. (Notice that if $|S| = |S_A|+1$ and $t(S)=t(S_A)+1$, then $\wh{S} = \wh{S}_A$ automatically holds.)
In this case, we have $k(S) =k(S_A)+1,$ and it follows that $D$ is not homologically $A$-adequate. This proves the first statement in the proposition, and the proof of the second statement on $B$-adequacy is similar.
\end{proof}

In summary then, for a link diagram $D$ on a surface $\Sigma$, it follows that
\begin{equation} \label{e-adeq-comp}
\text{plus-adequacy} \Longrightarrow \text{homological $A$-adequacy} \Longrightarrow \text{$A$-adequacy},
\end{equation}
with similar statements relating minus-adequacy, homological $B$-adequacy, and $B$-adequacy.

\begin{figure}[ht!]
\centering\includegraphics[scale = 0.8]{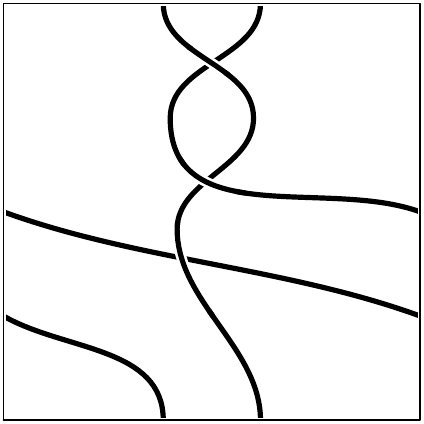}
\caption{An alternating diagram on the torus.} \label{f-3.7}
\end{figure}

In \Cref{ex-7knot}, we will see a knot diagram in a genus two surface which is adequate but not homologically adequate.
On the other hand, it is easy to construct examples which are homologically adequate but not simply adequate. For instance, consider the
alternating diagram $D$ with three crossings on the torus in \Cref{f-3.7}.
A straightforward calculation shows that it is homologically adequate but not simply adequate.
These examples show that none of the reverse implications in \eqref{e-adeq-comp} hold, therefore, the notion of adequacy in \Cref{defn:h-adequate} is strictly more general than either homological or simple adequacy.

In general, notice that
$${\rm span}(\la D\ra_\Sigma)\leq {\rm span}([D]_\Sigma) \leq 2 c(D) + 2t(S_A)+2t(S_B)\leq
2 c(D) + 2k(S_A)+2k(S_B),
$$
where ${\rm span}( \ \cdot \ )$ is the span in the $A$-degree. Therefore, \Cref{cor-adequate} immediately implies an analogous inequality holds for
homological adequacy, cf., \cite[Corollary 2.7]{Boden-Karimi-2019}.


\section{Alternating links and the Tait conjectures}
\label{s-Tait}

When tabulating knots, Tait formulated three conjectures on alternating links. The first one
states that any reduced alternating diagram of a classical link has minimal crossing number. The second one asserts that any two such diagrams representing the same link have the same writhe.
The third one states that any two reduced alternating diagrams of the same link are related by flype moves.
The first two conjectures were resolved almost 100 years later, independently by Kauffman, Murasugi, and Thistlethwaite, using the newly discovered Jones polynomial, \cite{Kauffman-87, Murasugi-871, Thistlethwaite-87}.
The third conjecture was established shortly after by Menasco and Thistlethwaite \cite{Menasco-Thistlethwaite-93}. The first two Tait conjectures actually hold more generally for adequate links \cite{Lickorish-Thistlethwaite}, and
their proofs have been generalized to homologically adequate links in thickened surfaces in \cite{Boden-Karimi-2019}.
Here, we generalize these results even further to adequate links in thickened surfaces.

Henceforth, all links in thickened surfaces will be unframed, unless stated otherwise.
Given an oriented link diagram $D$, let $c_{+}(D)$  be the numbers of crossings of type \KPXP,
and let $c_{-}(D)$ be the number of crossings of type \KPXN. The proof of the following theorem can be found in \Cref{s-proof-adequate-min}.

\begin{theorem} \label{t-adequate-min-cross}
Let $D$ and $E$ be oriented link diagrams on $\Sigma$ representing the same oriented unframed link in $\Sigma \times I$.
\begin{itemize}
\item[(i)] If $D$ is $A$-adequate, then $c_-(D)\leq c_-(E)$.
\item[(ii)] If $D$ is $B$-adequate, then $c_+(D)\leq c_+(E)$.
\end{itemize}
\end{theorem}

The {\bf crossing number} of a link $L\subset \Sigma\times I$, $c(L)$, is defined as the minimal crossing number among all diagram representatives of $L$.
A link $L\subset \Sigma\times I$ is said to be {\bf adequate} if it admits an adequate diagram on $\Sigma.$

Using \Cref{t-adequate-min-cross}, one can deduce the first and second Tait conjectures for adequate links in surfaces.
\begin{corollary}\label{c-Tait}
(i) Any adequate diagram of
a link in $\Sigma \times I$ has $c(L)$ crossings.\\
(ii) Any two adequate diagrams of the same oriented link in $\Sigma \times I$ have the same writhe.
\end{corollary}

\begin{proof}
Statements (i) and (ii) are immediate consequences of \Cref{t-adequate-min-cross}.
In the case of (ii), if adequate diagrams $D$ and $E$ represent the same oriented link, then
$c_+(D)=c_+(E)$ and $c_-(D)= c_-(E)$ by the above theorem and, hence,
$$w(D)=c_+(D)-c_-(D)=c_+(E)-c_-(E)=w(E). \qedhere$$
\end{proof}

\Cref{c-Tait} implies that
for an adequate link $L\subset \Sigma\times I$,
the writhe is a well-defined invariant of its oriented link type.

Let $g(\Sigma)$ be the sum of the genera of the connected components of $\Sigma.$
A link diagram $D$ on $\Sigma$ is {\bf minimally embedded} if it does not lie on a subsurface of $\Sigma$ of smaller genus. In other words, the complement of $D$ on $\Sigma$ has no non-separating loops. Let $N_D$ be a neighborhood of $D$ in $\Sigma$ small enough so that it is a ribbon surface
retractible onto $D$. A diagram $D$ is minimally embedded if and only if $g(N_D)=g(\Sigma).$

Furthermore, note that if $D$ is connected and $\Sigma$ is closed, then $D$ is minimally embedded if and only if $\Sigma\sm D$ is composed of disks. In that case, we say that $D$ is {\bf cellularly embedded}.

A link diagram $D$ on a closed surface $\Sigma$ is said to have {\bf minimal genus} if it is minimally embedded within its isotopy class.

In \cite{Manturov-2013}, it is proved that any cellularly embedded knot diagram with minimal crossing number has minimal genus. This result was recently extended to link diagrams, and the following is a restatement of Theorem 1 of \cite{BR-2021}.

\begin{theorem} \label{thm-min}
Any cellularly embedded link diagram with minimal crossing number has minimal genus.
\end{theorem}

A link diagram $D$ on $\Sigma$ is {\bf alternating} if, when traveling along any of its components, its crossings alternate between over and under.
A link $L\subset \Sigma\times I$ is {\bf alternating} if it can be represented by an alternating link diagram.

A crossing $x$ of $D$ is {\bf nugatory} if there is a simple loop in $\Sigma$
which separates $\Sigma$ and intersects $D$ only at $x.$

As observed in \cite{Boden-Karimi-2019}, although nugatory crossings in diagrams in $\Sigma=\R^2$ can always be removed by rotating one side of the diagram $180^\circ$ relative to the other, that is not always true for diagrams in non-contractible surfaces $\Sigma$, see \Cref{f-nugatory}.
A nugatory crossing is said to be {\bf removable} if the simple loop can be chosen to bound a disk, otherwise it is called {\bf essential}. A link diagram is {\bf reduced} if it does not contain any removable nugatory crossings.
For example, the knot in \Cref{f-7-crossing} contains an essential nugatory crossing.

\begin{figure}[!ht]
\centering\includegraphics[height=30mm]{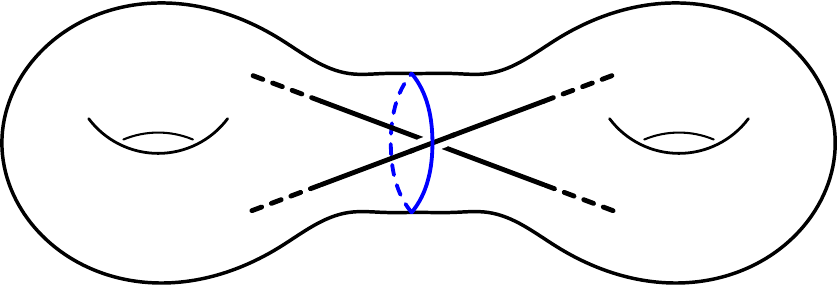}
\caption{An essential nugatory crossing.} \label{f-nugatory}
\end{figure}

The following strengthens Proposition 2.8 of \cite{Boden-Karimi-2019}.
Its proof is given in \Cref{s-proof-alt-adeq}.

\begin{theorem}\label{t-alt-adeq}
Any reduced alternating diagram is adequate.
\end{theorem}

Note that unlike Proposition 2.8 of \cite{Boden-Karimi-2019}, we do not assume here that $D$ is cellularly embedded, checkerboard colorable, nor that $D$ has no nugatory crossings.

A link diagram on $\Sigma$ is said to be {\bf weakly alternating} if
it is a connected sum $D_0\# D_1\# \cdots \# D_k$  of an alternating diagram $D_0$ in $\Sigma$ and  with alternating diagrams $D_1, \ldots, D_k$ in $S^2$ (cf., \Cref{l-connect-adeq}).
\Cref{t-alt-adeq} can be generalized to show that weakly alternating diagrams are adequate. In fact, in the next section we will prove \Cref{p-conn-sum-adeq}, showing that any diagram on a surface obtained as the connected sum of two adequate link diagrams is itself adequate.

Let us return to Tait conjectures now.
By \Cref{c-Tait}, any reduced alternating diagram $D$
has the minimal crossing number for all diagrams representing the same unframed link $L$ in $\Sigma\times I.$
Furthermore, all such oriented diagrams representing the same link $L$ have the same writhe.

The results of Kauffman, Murasugi, Thistlethwaite \cite{Kauffman-87, Murasugi-871, Thistlethwaite-87} imply that the span of the Kauffman bracket of any diagram $D\subset S^2$ satisfies
$${\rm span}([D]_{S^2})\leq 4c(D)+4,$$
or equivalently for the Jones polynomial,  that ${\rm span} (V_D(t)) \leq c(D)$, with
equality if $D$ is alternating.
Furthermore, in \cite{Thistlethwaite-87} Thistlethwaite proved that if $D\subset S^2$ is prime and non-alternating, then
$${\rm span}([D]_{S^2} )<4 c(D)+4.$$
In \cite{Turaev-1987}, it is observed that the above results hold if $D\subset S^2$ is weakly alternating, namely if $D$ is a connected sum of alternating diagrams.
Thus the Kauffman bracket $[ D ]_{S^2}$, together with $c(D)$, \emph{detect} weakly alternating classical links.

The homological Kauffman bracket of \cite{Boden-Karimi-2019} is not sufficiently strong to prove an analogous statement for links in thickened surfaces. Consider the two knots in the genus two surface in \Cref{f-4.107}.
These knots have the same homological Kauffman bracket, namely
$$\la D_1 \ra_{\Sigma}=\la D_2 \ra_{\Sigma}=3\delta z^2 - 4\delta^2 z+ (A^{4}+3+A^{-4})\delta,$$
but one of them is alternating and the other is not. Consequently, the homological Kauffman bracket does not detect alternating knots in thickened surfaces.

\begin{figure}[h!]
\centering
\includegraphics[height=42mm]{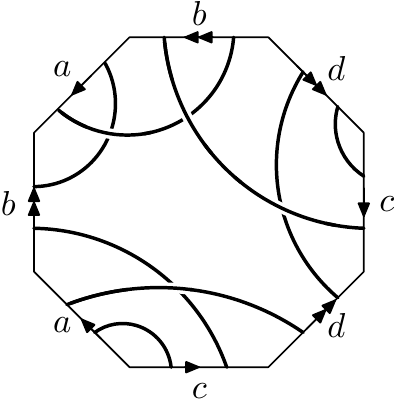} \qquad
\includegraphics[height=42mm]{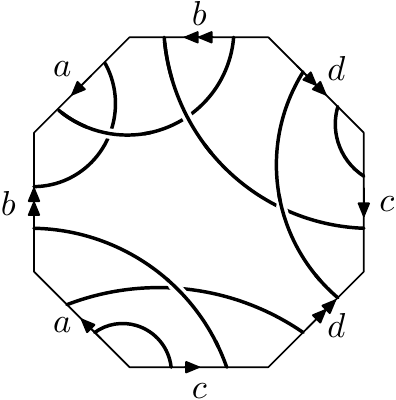}
\caption{Two knots in a genus two surface with the same homological Kauffman bracket.}\label{f-4.107}
\end{figure}

However, we are going to show that Kauffman, Murasugi, Thistlethwaite statements hold for the Kauffman bracket $[\, \cdot \, ]_\Sigma$ of diagrams in closed surfaces $\Sigma$ after replacing $4$ by $4|D|-4g(\Sigma)$ on the right.

Let $|D|$ denote the number of connected components of $D$ (which may be smaller than the number of connected components of the link in $\Sigma\times I$ represented by $D$).

Let $r(D)$ be the rank of the image of $i_* \colon H_1(D;\Z/2)\to H_1(\Sigma;\Z/2).$
If $D \subset \Sigma$ is minimally embedded, then $i_*$ is surjective and $r(D)=2g$.

The proof of the next result is given in \Cref{s-proof-altern-span}.

\begin{theorem} \label{t-altern-span}

(i) For any link diagram $D\subset \Sigma,$
$${\rm span}([D]_\Sigma) \leq 4c(D) + 4 |D|- 2 r(D).$$
(ii) If $D$ is cellularly embedded, reduced, and weakly alternating, then
$${\rm span}([D]_\Sigma) = 4c(D) + 4 |D|-4g(\Sigma).$$
(iii) If $D$ is not weakly alternating then
$${\rm span}([D]_\Sigma) < 4c(D) + 4 |D|-2r(D).$$
\end{theorem}

The assumptions of \Cref{t-altern-span} (ii) are necessary: 


If $D$ has a removable nugatory crossing, then eliminating it decreases the right hand side of the above equality but not the left hand side. Therefore, (ii) does not hold for diagrams with removable crossings.

It can also fail when $D$ is not cellularly embedded. For example, consider the alternating link in \Cref{f-minimaleg}. It has $t(S_A)=4$ and $t(S_B)=2$. Therefore, by \Cref{cor-adequate}, we have ${\rm span}([D]_\Sigma) \leq 16+12=28$, whereas $4c(D) + 4 |D|-4g(\Sigma) = 32.$ Note that this diagram is minimally embedded but not cellularly embedded.
\begin{figure}[h]
   \centering
   \includegraphics[width=3in]{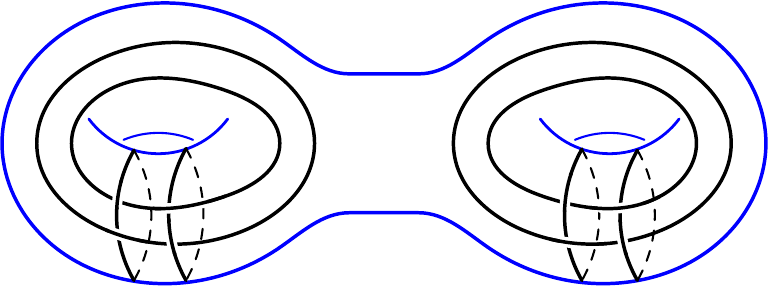}
   \caption{Minimally embedded alternating diagram for which the equality of \Cref{t-altern-span} (ii) does not hold.}
   \label{f-minimaleg}
\end{figure}

Although (ii) holds for weakly alternating diagrams, in the next section we will see that it does not hold generally for connected sums of alternating diagrams in arbitrary surfaces (see \Cref{ex-6knot}).

\begin{corollary} \label{cor-th}
Let $L$ be a  link in $\Sigma \times I$ with a reduced, weakly alternating diagram $D$ which is cellularly embedded. Then any other cellularly embedded diagram $E$ for $L$ satisfies $c(D)\leq c(E).$ If $E$ is not weakly alternating, then $c(D) < c(E).$ 
\end{corollary}

\begin{proof}
The first part of the proof is a direct consequence of Tait conjecture (\Cref{c-Tait}).
Let us prove the full statement now:
Any cellularly embedded link diagram on a connected surface is itself connected. Therefore, it is enough to prove the statement under the assumption that $\Sigma$ and $D$ are both connected.
\Cref{t-altern-span} (ii) then implies that $c(D)={\rm span}([D]_\Sigma)/4 +g(\Sigma)-1.$
If $E$ is a second link diagram for $L$ on $\Sigma$, then since $E$ is cellularly embedded, it must also be connected.
\Cref{t-altern-span} (i) implies that $$c(D) = {\rm span}([D]_\Sigma)/4 +g(\Sigma)-1 ={\rm span}([E]_\Sigma)/4 +g(\Sigma)-1 \leq c(E).$$
If $E$ is not weakly alternating, then \Cref{t-altern-span} (iii) shows the last inequality is strict, therefore, it follows that $c(D) < c(E)$.
\end{proof}

\begin{remark}
The corollary  gives an alternate proof of \Cref{thm-min} for non-split alternating links as follows.
Let $L$ be a non-split alternating link in $ \Sigma \times I,$ where $\Sigma$ is closed oriented surface, and let $D \subset \Sigma$ a minimal crossing cellularly embedded diagram for $L$.
Then \Cref{cor-th} implies that $D$ is an alternating diagram. The argument is completed by appealing to Proposition 6 of \cite{BK-2020}, which shows that alternating link diagrams have minimal genus.
\end{remark}

%
\section{Crossing number and connected sums}
\label{s-connected-sums}

In this section, we will study the behavior of the crossing number under connected sum of links in thickened surfaces.
This problem is closely related to an old and famous conjecture for classical links, which asserts that, for any two links $L_1, L_2$,
\begin{equation} \label{e-hard}
c(L_1 \# L_2) = c(L_1) + c(L_2).
\end{equation}
This conjecture has been verified for a wide class of links, including alternating links, adequate links, and torus links \cite{Diao-2004}.
Clearly, $c(L_1 \# L_2) \leq c(L_1) + c(L_2)$. In addition, in \cite{Lackenby-2009},
Lackenby has proved that, in general, one has a lower bound of the form:
$$c(L_1 \# L_2) \geq \tfrac{1}{152} \left( c(L_1) + c(L_2)\right).$$

The operation of connected sum is not so well-behaved for arbitrary links in thickened surfaces.

Just as for classical links, it depends on the choice of components which are joined as well as their orientations. However, unless one of the links is in $S^2 \times I$, it also depends on the diagram representatives as well as the choice of basepoints $x_i \in D_i$ where the link components are joined. The issue is the fact that a Reidemeister move applied to either of the link diagrams may change the link type of their connected sum.  We take a moment to quickly review its construction.

Suppose $\Sigma_1$ and $\Sigma_2$ are oriented surfaces and let $\Sigma_1\# \Sigma_2$ denote their connected sum. It is obtained from the union $(\Sigma_1 \sm {\rm int}\,B_1) \cup (\Sigma_2 \sm {\rm int}\,B_2)$ by gluing $\p B_1\subset \Sigma_1$ to
$\p B_2\subset \Sigma_2$ by an orientation reversing homeomorphism $g \colon \p B_1 \to \p B_2$.
For connected surfaces, $\Sigma_1\# \Sigma_2$ is independent of the choice of disks $B_i \subset \Sigma_i$ and gluing map.

If $D_1 \subset \Sigma_1$ and $D_2 \subset \Sigma_2$ are link diagrams, we can choose cutting points $x_i\in D_i$ and disk neighborhoods $B_i$ from $\Sigma_i$ such that $B_i\cap D_i$ is an interval for $i=1,2$. Then the surface $\Sigma_1\# \Sigma_2$ can be formed in such a way that  $D=(D_1 \sm {\rm int}\,B_1) \cup (D_2 \sm {\rm int}\, B_2)$  is a link diagram in $\Sigma_1\# \Sigma_2$.
If $D_1,D_2$ are oriented link diagrams, then we require the gluing to respect the orientations of the arcs.
The resulting diagram is called a {\bf connected sum} of $D_1$  and $D_2$.
In general, it depends on the choice of link diagrams $D_1,D_2$, components being joined, and the points $x_i \in D_i.$ However, it is independent of the choice of disk neighborhoods $B_i$ containing $x_i$.

The next result shows that when one of the diagrams lies in $S^2\times I$, the operation of connected sum
is well-behaved.

\begin{lemma} \label{l-connect-adeq}
Let $D_1 \subset \Sigma\times I$ and $D_2\subset S^2\times I$ be
oriented diagrams, where $\Sigma$ is an arbitrary surface.
Then the connected sum of $D_1$ and $D_2$ is independent of the choice of the cutting points $x_1, x_2$ on the selected components of $D_1$ and of $D_2.$

We will denote the connected sum in this case by $D_1\# D_2$. The oriented link type of $D_1\# D_2$ depends only on the link types of $D_1$ and $D_2$ and a choice of which components are joined.
\end{lemma}

\begin{proof} One can shrink the image of $D_2$ in the connected sum so that all its crossings lie in a small 3-ball $B^3$ in $\Sigma \times I.$ By an isotopy, we can move the ball along arcs of $D_1$ representing the component to which $D_1$ is joined, and moving over or under the other arcs at any crossing that we encounter.

This shows that the connected sum is independent of the choice of the cut point $x_1$ on $D_1$. The independence on the cut point $x_2$ on $D_2$ follows from the well-known fact that all long knots, or rather $(1,1)$ tangles, obtained by cutting $D_2$ at different points $x_2$ of its specified component are isotopic (as $(1,1)$ tangles). Shrinking $D_2$ into a small 3-ball also allows one to translate any Reidemeister move of $D_1$ or $D_2$ into a Reidemeister move on the connected sum $D_1 \#D_2$. This proves the last statement.
\end{proof}

\begin{proposition}\label{p-conn-sum-adeq}
Any connected sum of  two $A$- or $B$-adequate diagrams is itself $A$- or $B$-adequate (respectively).
\end{proposition}

\begin{proof}
Let $D$ be a link diagram in $\Sigma_1 \# \Sigma_2$ obtained as the connected sum of $A$-adequate diagrams
$D_1\subset \Sigma_1$ and $D_2\subset \Sigma_2$, and suppose to the contrary that $D$ is not $A$-adequate.
By \Cref{p-alt-adeq}, there is a state $S$ for $D$ adjacent to $S_A$ with $t(S,\Sigma_1\#\Sigma_2)=t(S_A, \Sigma_1\#\Sigma_2)+1$ and $|\wh{S}|=|\wh{S}_A|$ in $\Sigma_1 \# \Sigma_2$. In particular, $|S|=|S_A|+1$, and the transition from $S_A$ to $S$ involves one cycle of $S_A$ splitting into two cycles.

Let $x$ be the crossing of $D$ where the smoothing is changed in the transition from $S_A$ to $S$. We can assume, without loss of generality, that $x$ is a crossing from $D_1$.
Let $C$ be the cycle of $S_A$ that splits into two cycles, $C'$ and $C''$ under this transition.
Since  $t(S,\Sigma_1\#\Sigma_2)=t(S_A, \Sigma_1\#\Sigma_2)+1$, one of
the cycles $C'$ and $C''$, say $C'$, must be trivial.

If $C$ is a cycle contained in $S_A(D_1)$,  then the same is true for $C'$ and $C''$.
However, this contradicts the assumption that $D_1$ is $A$-adequate.

Otherwise, $C=C_1\# C_2$ must be a connected sum of a cycle
$C_1$ in $S_A(D_1)$ with a cycle $C_2$ in $S_A(D_2).$
In the transition from $S_A$ to $S$, the cycle $C_1\# C_2$ splits into $C_1'\# C_2$ and $C_1''\# C_2$.
Further, since $C'=C_1'\# C_2$ is trivial, it follows that $C_1'$ must be trivial in $\Sigma_1$ and $C_2$ must be trivial in $\Sigma_2$.

If $C_1\# C_2$ is trivial, then $C_1''\# C_2$ must also be trivial. That would imply that all three of $C_1, C_1',C_1''$ are trivial in $\Sigma_1$. This again contradicts the assumption that $D_1$ is $A$-adequate, and we take a moment to explain this point.

Let $S(D_1)$ be the corresponding state for $D_1$. It is obtained from $S_A(D_1)$ by switching the smoothing at $x$.
The transition from $S_A(D_1)$ to $S(D_1)$ involves  $C_1$ splitting
into  $C_1'$ and $C_1''$. Since all three of $C_1, C_1',C_1''$ are trivial in $\Sigma_1$, we have $t(S(D_1)) = t(S_A(D_1))+1$ and $|\wh{S}(D_1)| = \wh{S}_A(D_1)$ in $\Sigma_1,$ which contradicts the assumption of $A$-adequacy of $D_1$.

The other possibility is that $C_1\#C_2$ is non-trivial. Since $C_2$ is trivial in $\Sigma_2$,
the cycles $C_1$ and $C_1''$ must both be nontrivial in $\Sigma_1$. The transition from $S_A(D_1)$ to $S(D_1)$ still involves $C_1$ splitting into $C_1'$ and $C_1''$, only now $C_1, C_1''$ are nontrivial and $C_1'$ is trivial in $\Sigma_1.$ Thus $t(S(D_1)) = t(S_A(D_1))+1$ and $|\wh{S}(D_1)| = \wh{S}_A(D_1)|$ in $\Sigma_1,$ which again contradicts the assumption of $A$-adequacy of $D_1$.
Therefore, $D=D_1\#D_2$ must be $A$-adequate.

The proof of $B$-adequacy of $D$ is similar.
\end{proof}

\begin{corollary}\label{c-connected}
Suppose $L_1 \subset \Sigma_1 \times I$ and $L_2 \subset \Sigma_2 \times I$ are links represented by adequate diagrams $D_1 \subset \Sigma_1$ and $D_2 \subset \Sigma_2.$
Then any  link $L$ in $(\Sigma_1 \#\Sigma_2)\times I$ admitting a diagram which is a connected sum of $D_1$ and $D_2$ is
itself adequate. Further, the crossing number and writhe satisfy
$c(L) = c(L_1)+c(L_2)$ and $w(L)=w(L_1)+w(L_2)$.
\end{corollary}

\begin{proof}
Suppose $L$ is represented by $D=D_1 \#D_2 \subset \Sigma_1\#\Sigma_2$. Then $D$ is adequate by \Cref{p-conn-sum-adeq}. Further, by parts (i) and (ii) of \Cref{c-Tait}, we see that:
\begin{eqnarray*}
c(L)&=& c(D)=c(D_1)+c(D_2)= c(L_1)+c(L_2) \quad \text{and} \\
w(L)&=& w(D)=w(D_1)+w(D_2)= w(L_1)+w(L_2). \qedhere
\end{eqnarray*}
\end{proof}

\begin{example} \label{ex-6knot}
\Cref{f-connected-sum} shows a knot diagram $D$ in the genus two surface obtained as the connected sum of two alternating diagrams of the same knot in the torus. One can easily verify that $D$ is reduced and cellularly embedded, but not alternating. Further, \Cref{p-conn-sum-adeq} implies that this diagram is adequate, and therefore a minimal crossing diagram for the knot type. Direct calculation reveals that $t(S_A) = 2, t(S_B) = 0,$ and $|\hat{S}_A| = |\hat{S}_B|=1$. Therefore,
${\rm span}([ D ]_\Sigma )=16$. On the other hand, since $4(c(D)+|D|-g(\Sigma)) = 20,$
by \Cref{t-altern-span} (ii), it follows that
$D$ is not weakly alternating, and in fact not equivalent to any weakly alternating knot in $\Sigma \times I$.

\begin{figure}[!ht]
   \centering
   \includegraphics[width=3in]{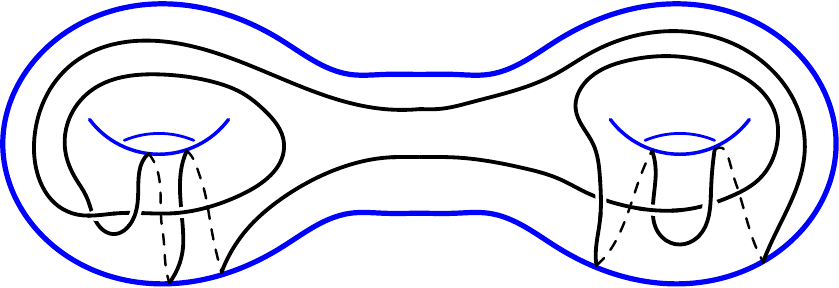}
   \caption{A connected sum of alternating diagrams.}
   \label{f-connected-sum}
\end{figure}
\end{example}

\begin{example} \label{ex-7knot}

\Cref{f-7-crossing} shows a knot in a genus two surface with an essential nugatory crossing.
Since it is reduced and alternating, \Cref{t-alt-adeq} shows that it is adequate.  Note that this diagram is not homologically adequate. In fact, if $S$ is the state with a $B$-smoothing at the nugatory crossing and $A$-smoothings at all the other crossings, then one can show that $|S| = |S_A|+1$ and $k(S)>k(S_A).$

Notice that this knot can also be obtained as the connected sum of two alternating knots $K_1,K_2$ in $T^2 \times I$, with $c(K_i)=3$ but after performing a Reidemeister one move on one of them to obtain a diagram with four crossings. In particular, this example shows that a connected sum of two diagrams $D_1 \subset \Sigma_1$ and $D_2 \subset \Sigma_2$ can be adequate even when one of them is not adequate.

\begin{figure}[ht!]
\centering\includegraphics[width=3in]{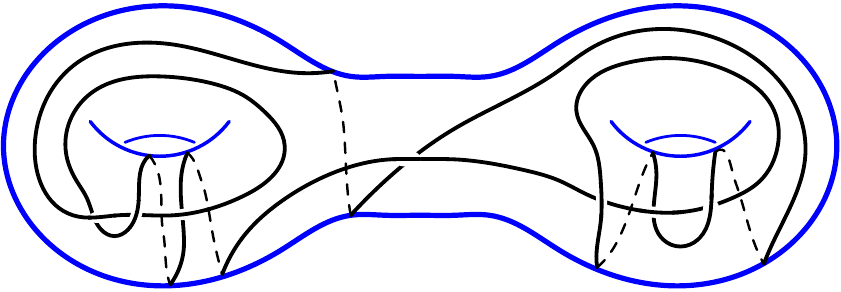}
\caption{An alternating diagram with an essential nugatory crossing.} \label{f-7-crossing}
\end{figure}
\end{example}

Suppose $L_1 \subset \Sigma_1 \times I$ and $L_2 \subset \Sigma_2 \times I$ are two alternating links in thickened surfaces with
$g(\Sigma_i)>0$ for $i=1,2$. Suppose further that $D_i$ is a link diagram on $\Sigma_i$ representing $L_i$ for $i=1,2$, and that $D_1,D_2$ are both reduced and alternating.

Instead of forming the connected sum of $D_1$ and $D_2$, take one of the diagrams and insert an arbitrary number (say $n$) of twists before forming the connected sum. See \Cref{f-twist-connect} for an illustration.

The result will be a diagram $D$, which is similar to a connected sum of $D_1$ and $D_2$, but with $n$ essential nugatory crossings in between. This construction can be carried out so that $D$ is reduced and alternating. In particular, it will have crossing number $c(D) =c(D_1)+c(D_2)+n$. If $L$ denotes the link type of $D$, and since $D_1$ and $D_2$ are alternating and have minimal crossing number, this shows that the analogue of \eqref{e-hard} can fail arbitrarily badly for links in thickened surfaces other than $S^2 \times I$.

\begin{figure}[ht!]
\centering\includegraphics[width=4.8in]{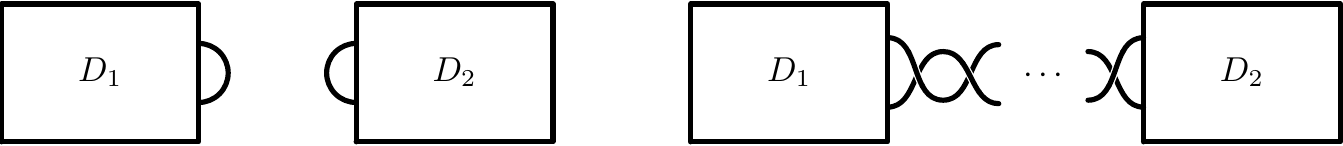}
\caption{Adding twists to a connected sum to create essential nugatory crossings.} \label{f-twist-connect}
\end{figure}

The reason \eqref{e-hard}  fails in general for connected sums of links in thickened surfaces is due to the use of non-minimal diagrams in forming the connected sum. However, if one restricts the connected sum operation to minimal crossing diagrams, then one gets a plausible generalization:

\begin{conjecture}\label{c-connected-sum}
Suppose $L_1 \subset \Sigma_1 \times I$ and $L_2 \subset \Sigma_2 \times I$ are links in thickened surfaces with minimal crossing representatives $D_1, D_2$, respectively. Then any link $L$ in the thickening of $\Sigma_1\# \Sigma_2$ arising as a connected sum of $D_1$ and $D_2$ satisfies
$$c(L) = c(L_1) + c(L_2).$$
\end{conjecture}

Note that the assumption that $D_1,D_2$ are minimal crossing representative
implies immediately that $$c(L) \leq c(L_1)+c(L_2).$$
In fact, the inequality may fail without that assumption. This is related to the fact that crossing number is not additive under connected sum for virtual knots. For example, the Kishino knot is the connected sum of two virtual unknots.
As evidence, notice that \Cref{c-connected} confirms that the conjecture is true if $L_1$ and $L_2$ are adequate links in thickened surfaces. In particular, it holds for alternating and weakly alternating links.

\section{Proofs of Theorems \ref{t-adequate-min-cross}, \ref{t-alt-adeq}, and \ref{t-altern-span}}
\subsection{Proof of \Cref{t-adequate-min-cross}}
\label{s-proof-adequate-min}

Given a link diagram $D$ on $\Sigma$ and positive integer $r$, the
{\bf $r$-th parallel} of $D$ is the link diagram $D^r$ on $\Sigma$ in which each link component of $D$ is replaced by $r$ parallel copies, with each one repeating the same ``over'' and ``under'' behavior of the original component.

\begin{lemma}
If $D$ is  $A$-adequate, then $D^r$ is also $A$-adequate. If $D$ is $B$-adequate, then $D^r$ is also  $B$-adequate.
\end{lemma}

\begin{proof}
Let $S_{A}(D)$ and $S_{A}(D^r)$ be the pure $A$-smoothings of $D$ and the pure $A$-smoothings of $D^r$, respectively. It is straightforward to check that $S_{A}(D^r)$ is the $r$-parallel of $S_{A}(D)$.

Suppose $D^r$ is not  $A$-adequate. Then there is a state $S'$ obtained by switching one $A$-smoothing in $S_{A}(D^r)$ to a $B$-smoothing, such that $t(S_{A}(D^r))<t(S')$, and
$\wh{S}_{A}(D^r)=\wh{S'}$.
In the terminology of the proof of \Cref{p-alt-adeq}, that can only happen for a smoothing change of type (i). More specifically, when the smoothing change involves one of the innermost cycles in $S_{A}(D^r)$ which is self-abutting and which, when split, creates a new trivial cycle in $S'$. That is only possible if there is a self-abutting cycle in $S_{A}(D)$ which, when split, creates a new trivial cycle. Since $D$ is  $A$-adequate, this cannot happen.

An analogous argument proves the statement for  $B$-adequate diagrams.
\end{proof}

\noindent{\it Proof of \Cref{t-adequate-min-cross}.}
\noindent (i) Since
$$c(D)-w(D)=c_{+}(D)+c_{-}(D)-(c_{+}(D)-c_{-}(D))= 2c_-(D),$$
we will prove that
$$c(D)-w(D)\leq c(E)-w(E).$$
Our argument is an adaptation of Stong's proof \cite{Stong-1994} (cf., Theorem 5.13 \cite{Lickorish}).

Let $L_1, \ldots, L_m$ be the components of $L$ and let $D_i$ and $E_i$ be the subdiagrams of $D$ and $E$ corresponding to $L_i$. For each $i =1,\ldots, m,$ choose non-negative integers $\mu_i$ and $\nu_i$ such that $w(D_i)+\mu_i=w(E_i)+\nu_i$. Let $D'$ be composed of components $D_1',\ldots, D_m'$, where each
$D'_i$ is obtained from $D_i$ by adding $\mu_i$ positive kinks to it. (These kinks do not cross with other components).
Similarly, let $E'$ be composed of components $E_1',\ldots, E_m'$, where each $E'_i$ is obtained from $E_i$ by adding $\nu_i$ positive kinks to it. Notice that $D'$ is still  $A$-adequate.

The writhes of the individual components satisfy:
$$w(D'_{i})=w(D_i)+\mu_i=w(E_i)+\nu_i=w(E'_{i}).$$
Further, the sum of the signs of the crossings of $D_i'\cap D_j'$ coincides with the sum of the signs of the crossings of $E_i'\cap E_j'$, since both are  equal to the linking number of $L_i$ and $L_j$. Thus $w(D')=w(E')$.

For any $r$, consider the $r$-th parallels $(D')^r$ and $(E')^r$ now. Then $w((D')^r)=r^2 w(D')$, because  each crossing of $D'$ corresponds to $r^2$ crossings in $(D')^r$ of the same sign. The diagrams $(D')^r$ and $(E')^r$, are equivalent and have the same writhe, thus their  Kauffman brackets must be equal. In particular, we have $d_{\max}([ (D')^r]_\Sigma)=d_{\max}([ (E')^r ]_\Sigma)$. \Cref{p-dminmax} implies now that
\begin{eqnarray*}
d_{\max}([(D')^r ]_\Sigma)&=& \left(c(D)+\sum_{i=1}^m \mu_i\right)r^2+2\left(t(S_{A}(D))+\sum_{i=1}^m \mu_i\right)r,\\
d_{\max}([(E')^r ]_\Sigma)&\leq& \left(c(E)+\sum_{i=1}^m \nu_i\right)r^2+2\left(t(S_A(E))+\sum_{i=1}^m \nu_i \right)r.
\end{eqnarray*}
Since this is true for all $r$, by comparing coefficients of the $r^2$ terms,  we find that:
\begin{equation}\label{eqn:compare}
 c(D)+\sum_{i=1}^m \mu_i \leq c(E)+\sum_{i=1}^m \nu_i.
 \end{equation}
Subtracting $\sum_{i=1}^m (\mu_i +w(D_i))= \sum_{i=1}^m (\nu_i + w(E_i))$ from both sides of \eqref{eqn:compare}, we get that
\begin{equation}\label{eqn:new}
c(D)-\sum_{i=1}^m w(D_i)\leq c(E)-\sum_{i=1}^m w(E_i).
\end{equation}
Subtracting the total linking number of $L$ from both sides of \eqref{eqn:new} gives the desired inequality.

\noindent The proof of (ii) is analogous. One adds negative kinks to $D$ and $E$ in this case.
\qed

%
\subsection{Proof of \Cref{t-alt-adeq}}
\label{s-proof-alt-adeq}

A link diagram $D$ on $\Sigma$ is {\bf alternable} if it can be made alternating by inverting some of its crossings. Every classical link diagram is alternable, but the same is not true for link diagrams in arbitrary surfaces. For example, the knot diagram in the torus in \Cref{f-2.1} is not alternable.

\begin{figure}[!ht]
\centering
\includegraphics[height=30mm]{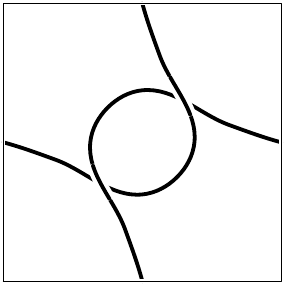}
\caption{A knot diagram in the torus which is not alternable.} \label{f-2.1}
\end{figure}

A link diagram $D$ on $\Sigma$ is {\bf checkerboard colorable} if the components of $\Sigma\sm D$ can be colored by two colors such that any two components of $\Sigma\sm D$ that share an edge have opposite colors.

\begin{proposition}\label{p-alter-cc}
Any minimal embedding $D$ on $\Sigma$ is alternable if and only if it is checkerboard colorable.
\end{proposition}

\begin{proof}
Observe that filling the boundaries of $\Sigma$ with disks does not affect alternability or checkerboard colorability.Likewise, removing disks from $\Sigma\sm D$ also does not affect alternability or checkerboard colorability. This  has two consequences:

(a)  It is enough to prove this statement for surfaces $\Sigma$ with all boundary components capped, i.e., for closed surfaces.

(b) Since Kamada proved that if a diagram $D$ is a deformation retract of $\Sigma$ then it is alternable if and only if it is checkerboard colorable, \cite[Lemma 7]{Kamada-2002},
our statement holds for cellularly embedded diagrams.

Our strategy is to reduce the proof to this case of cellular embeddings. Suppose that $C$ is a non-disk component of $\Sigma \sm D$. Then it contains a non-contractible simple closed loop $\alpha$. Let $\Sigma'$ be obtained by cutting  $\Sigma$ along $\alpha$ and by capping the boundary components. The loop $\alpha$ must be separating $\Sigma$, since otherwise $D\hookrightarrow \Sigma'$ would be a lower genus embedding of $D.$ Observe now that since $\Sigma$ is a connected sum of two surfaces $\Sigma_1\# \Sigma_2$, where $\Sigma_1\cup \Sigma_2=\Sigma'$ and $D$ is a disjoint union of $D\cap \Sigma_1$ and of $D\cap \Sigma_2$, it is enough to prove that $D\subset \Sigma_i$ is checkerboard colored for $i=1,2.$

By repeating this process as long as possible, we reduce the statement to cellularly embedded diagrams, which is covered by (b) above.
\end{proof}

\begin{lemma}\label{lem-24}
Any alternable diagram can be extended by disjoint simple closed loops to a checkerboard colorable one.
\end{lemma}

\begin{proof}
The surface $N_D\subset \Sigma$, being a regular neighborhood of $D$, is checkerboard colorable by the earlier mentioned result of Kamada,  \cite[Lemma 7]{Kamada-2002}. The only reason that coloring does not extend to $D\subset \Sigma$ is that some connected components $C$ of $\Sigma \sm {\rm int}\, N_D$ may have multiple connected components of their boundary whose neighborhoods are colored differently. However, that issue can be resolved by adding simple closed loops around those boundary components of $C$ which are white.
\end{proof}

\noindent{\it Proof of \Cref{t-alt-adeq}:} Let $D$ be alternating diagram without removable crossings. By \Cref{lem-24}, by adding disjoint simple closed loops to $D$ we obtain a diagram $D'$ which is alternating and checkerboard colorable. Hence, it is enough to prove that $D'$ is adequate. Let us assume for simplicity of notation that $D$ is checkerboard colorable.

We will prove the $A$-adequacy of $D$ only, as the proof of $B$-adequacy is identical. Let $S$ be a state  with all $A$-smoothings except for a $B$-smoothing at a crossing $x$ of $D$. We will prove that $D$ is  $A$-adequate ``at $x$,'' meaning that $t(S)\leq t(S_A)$ or $\wh{S}\ne \wh{S}_A$ in $\sS(\Sigma \times I).$

As in the proof of \Cref{p-alt-adeq},
there are three cases, and to check adequacy, it is enough to check that
the conditions of \Cref{defn:h-adequate} hold in the first case, namely
when $|S| =|S_A|+1.$ Therefore, $S_A$ must contain a self-abutting cycle $C$, and in the transition from $S_A$ to $S$, the cycle $C$ splits into two cycles $C_1,C_2$ of $S$.
Since $D$ is alternating and checkerboard colorable, $S_A$ bounds a subsurface $\Sigma'$ of $\Sigma$ of a certain color, say white, which contains no crossings of $D$.

\begin{figure}[h]
   \centering
   \includegraphics[width=2in]{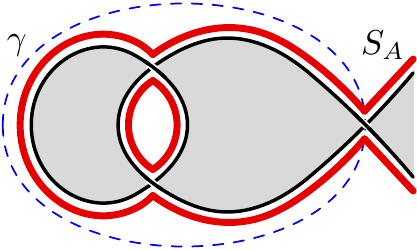}
   \caption{}
   \label{fig:alternating}
\end{figure}

We claim that neither $C_1$ nor $C_2$ is trivial. Indeed, if say $C_1$ were trivial, then there would be a loop $\gamma$ parallel to $C_1$ totally inside $\Sigma'$ except for a little neighborhood of $x$, in which it would cross $x$. Such a curve would imply that the crossing $x$ is removable,  (see for example \Cref{fig:alternating}), which is a contradiction. Therefore, neither $C_1$ nor $C_2$ is trivial, and it follows that $t(S)= t(S_A).$ Therefore, $D$ is $A$-adequate at $x$, and this completes the proof of the theorem.
\qed

%
\subsection{Link diagrams and shadows}
\label{s-shadow}

A {\bf link shadow} in $\Sigma$ is a $4$-valent graph in $\Sigma$, possibly with loop components. In other words, a shadow is a link diagram with crossing types ignored. For that reason we refer to shadow vertices as crossings and the components of any link realization of a shadow as its link components. (Not to be confused with connected components of a shadow.)

Some properties of link diagrams are entirely determined by its link shadow. For example, we will say that   a link shadow $D$ on $\Sigma$ is {\bf checkerboard colorable} if the components of $\Sigma\sm D$ can be colored by two colors such that any two components of $\Sigma\sm D$ that share an edge have opposite colors. Clearly, a link diagram is checkerboard colorable if and only if its link shadow is. Similarly, a link shadow is {\bf minimally embedded} if it does not lie in a subsurface of $\Sigma$ of smaller genus, and it is immediate that a link diagram on $\Sigma$ is minimally embedded if and only if its link shadow is.

Each shadow crossing has two smoothings, which cannot be differentiated as $A$- and $B$-type, as in the case of link diagrams. For that reason, for link shadows it is customary to place markers at the crossings indicating the smoothing as in \Cref{marker}.

\begin{figure}[!ht]
\centering
\includegraphics[height=23mm]{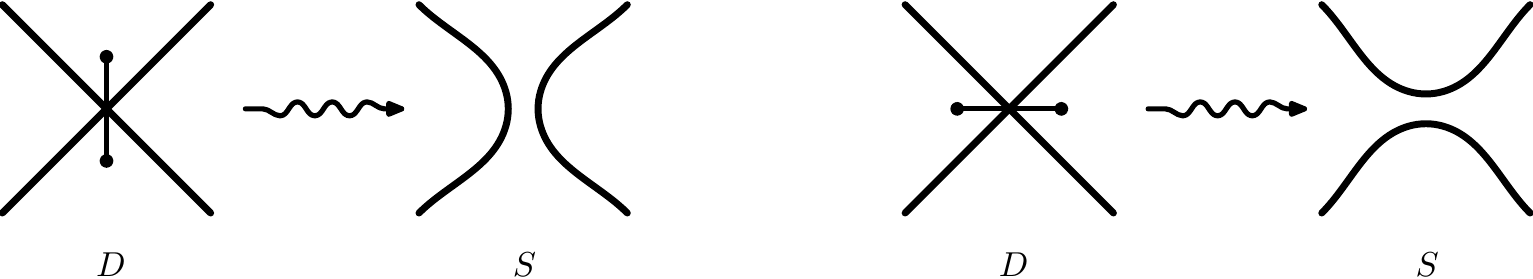}\quad
 \caption{Two types of markers for a state of a link shadow.}
\label{marker}
\end{figure}

Two consecutive crossings can have identical or opposite smoothings, see \Cref{alt-markers}. An {\bf alternating state} of a shadow is one with alternating crossing smoothings along all of its link components. In other words, a state is alternating if the smoothings at every pair of consecutive crossings are opposite.

Not all link shadows admit alternating smoothings, for example the shadow of the non-alternable knot in the torus in \Cref{f-2.1}. On the other hand, any link shadow of an alternating link diagram admits two alternating smoothings, namely the shadow smoothings coming from $S_A$ and $S_B$.

Given a state $S$ for a link shadow $D$, the dual state is denoted $S^{\vee}$ and has opposite smoothing to $S$ at each crossing of $D$. Notice that a state $S$ is alternating if and only if its dual state $S^{\vee}$ is alternating.

We say that a $2$-disk $D^2$ is $2$-cutting, or simply, {\bf cutting} a shadow $D$ if its boundary intersects $D$ transversely at two points (which are not crossings) and $D^2\cap D$ contains some but not all the crossings of $D$. A connected shadow $D$ is said to be {\bf strongly prime} if it has no cutting disk. More generally, a shadow $D$ is strongly prime if all of its connected components are.

\begin{figure}[!ht]
\centering
\includegraphics[height=20mm]{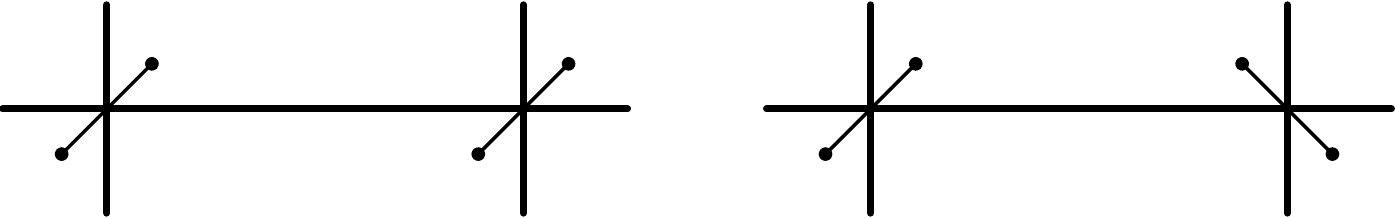}\quad
 \caption{Two consecutive crossings with identical markers (left) and opposite markers (right).}
\label{alt-markers}
\end{figure}

\begin{lemma}\label{l-strongly-prime}
Every crossing of every strongly prime shadow $D\subset \Sigma$ has at least one smoothing producing a shadow which is again strongly prime. If $D$ is connected, then the smoothing can be chosen so the resulting shadow is connected and strongly prime.
\end{lemma}

For classical links, a proof of this statement can be found in \cite{Lickorish}. That proof relies on checkerboard colorability of the diagram, which is of course true for classical links. Below, we give a proof that does not require the shadow to be checkerboard colorable.

\begin{proof}
Without loss of generality we can assume that $D$ is connected. Assume now that the smoothings of a crossing $v$ in a strongly prime $D$ produce diagrams $D_1, D_2$ neither of which is strongly prime. Let $B_1,B_2$ be cutting disks for $D_1$ and $D_2.$  Since $D$ is strongly prime, we can assume that $v\in \p B_i$ for $i=1,2$. We can also assume that $\p B_1$ and $\p B_2$ are in transversal position. Let $C$ be the connected component of $B_1\cap B_2$ containing $v$, as in \Cref{f-sprime1} (left). The circles $\p B_1,$ $\p B_2$ are broken because they may intersect each other many times.

By modifying $B_1$ or $B_2$ slightly if necessary we can assume that $D$ does not contain the second intersection point, $w$, of $\p B_1\cap \p B_2$ in $C$.

Let $\alpha_1={\rm int}(C\cap \p B_1)$ and $\alpha_2={\rm int}(C \cap \p B_2)$. (Note that $v\not\in\alpha_1\cup\alpha_2$.) Since $D$ intersects $\p B_i-\{v\}$ twice, for $i=1,2,$ and $D$ intersects $\alpha_1\cup \alpha_2$ at an odd number of points, we have the following possibilities: \\
(1) $|D\cap \alpha_2|=1,$ $D\cap \alpha_1= \varnothing$.\\
(2) $|D\cap \alpha_2|=2,$ $|D\cap \alpha_1|= 1$.\\
(3) one of the two cases above with $\alpha_1$ interchanged with $\alpha_2.$ We will ignore it without loss of generality.

\begin{figure}[!ht]
\centering
\includegraphics[height=28mm]{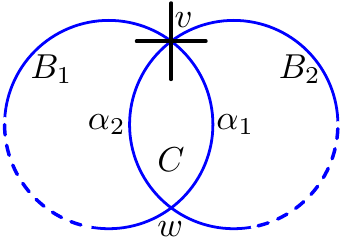}\quad
\includegraphics[height=28mm]{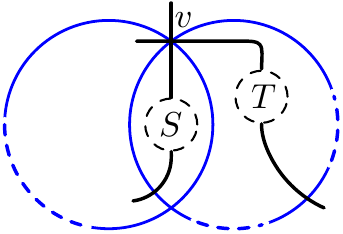}\quad
\includegraphics[height=28mm]{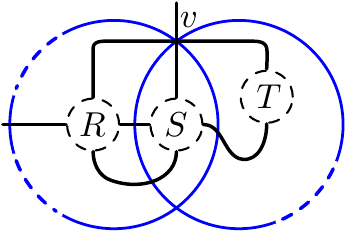}
\caption{}
\label{f-sprime1}
\end{figure}

In the first case, $D$ looks like in \Cref{f-sprime1} (center), where $S,T$ (in dashed circles) are shadow tangles.  In that case, since neighborhoods of $S, T$ are not cutting disks for $D$, the tangles $S,T$ are crossingless. That means that $B_2$ is not a cutting disk for $D_2$ --  a contradiction.

In the second case, $D$ looks like in \Cref{f-sprime1} (right), where $R,S,T$ are shadow tangles. Note that all crossings of $D$, other than $v$, are contained in $R, S$ or $T$, since otherwise a disk containing $v$, $R, S$ and $T$ but no other crossings of $D$ would be cutting for $D$. Note also that, as in the first case, $T$ is crossingless. That means that all crossings of $D_1$ are in $R$ and $S$. Hence, $B_1$ is not cutting for $D_1$ -- a contradiction.

Now assume that $D$ is connected. Then for one of the smoothings of $D$ at $v$ will be connected. Let $D'$ denote the connected shadow obtained from smoothing $D$, and assume the other smoothing is disconnected. We claim that $D'$ is strongly prime.

Assume to the contrary that $D'$ is not strongly prime. Then there is a cutting disk $B$ containing some but not all the crossings of $D'$ (see \Cref{f-newpic}). We can assume that $x\in \p B$ and that $D'$ is obtained by the smoothing of $x$ tangential to $\p B.$ However, since the other smoothing of $D$ at $v$ is disconnected, the strands from the tangles $R$ and $T$ cannot cross each other. The neighborhoods of $R,T$ give cutting disks for $D$ unless the tangles $R,T$ are crossingless, but then $B$ would not be a cutting disk for $D'$,  which is a contradiction.
\end{proof}

\begin{figure}[!ht]
\centering
\includegraphics[height=28mm]{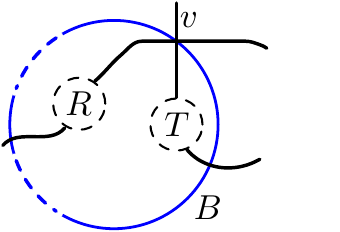}
\caption{A cutting disk for $D'$.}
\label{f-newpic}
\end{figure}

Suppose $D \subset \Sigma$ is a link shadow. Let $N_D$ denote a neighborhood of $D$ in $\Sigma$ small enough so that it is a ribbon surface retractible onto $D$. A \textbf{local checkerboard coloring} of $D$ is a checkerboard coloring of $D\subset N_D$. If one exists, we say that $D$ is \textbf{locally checkerboard colorable}. (The pair $(D, N_D)$ is the shadow of an abstract link diagram \cite{Kamada-2000}, or ALD for short. This condition is equivalent to saying that $(D, N_D)$ is the shadow of a checkerboard colorable ALD.)

Obviously, if $D \subset \Sigma$ is checkerboard colorable, then it is locally checkerboard colorable. The converse holds if $D\subset \Sigma$ is cellularly embedded, but in general, a shadow can be locally checkerboard colorable without being checkerboard colorable.

\begin{lemma}\label{alt-check}
Suppose $D \subset \Sigma$ is a link shadow. Then $D$ is locally checkerboard colorable if and only if it admits an alternating state.
\end{lemma}

\begin{proof}
If $D$ is locally checkerboard colorable, then let $S$ be the state whose  smoothings at each crossing join the white regions. Then $S$ is an alternating state.

Conversely, suppose $S$ is an alternating state. Let $\wh{\Sigma}$ be the surface obtained from $N_D$ by attaching disks to each of its boundary component. Then $D \subset \wh{\Sigma}$ is cellularly embedded. We can color $\wh{\Sigma} \sm D$ so that each cycle in $S$ bounds a black disk and each cycle in $S^\vee$ bounds a white disk. To see this, notice that at each smoothing of $S$, two local regions are joined. We can color the joined regions white and extend the coloring to the rest of $\wh{\Sigma} \sm D$. This determines a local checkerboard coloring of $D$.
\end{proof}

If $S$ and $S'$ are adjacent states on a shadow $D$ with $|S'|=|S|,$ then the transition from $S$ to $S'$ is called a \textbf{single cycle bifurcation}.

\begin{lemma}\label{cc-single}
A connected shadow $D$ is locally checkerboard colorable if and only if there is no single cycle bifurcation in its cube of resolutions.
\end{lemma}

\begin{proof}
For one implication, we apply \cite[Proposition 5.11]{Karimi-2019} to see that if $D$ is locally checkerboard colorable, then its cube of resolutions does not contain any single cycle bifurcations.

The other implication is proved by induction on the crossing number. To start, we verify it for 1-crossing shadows, which can be classified into the {\bf first type} $\ \diag{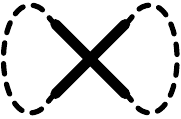}{0.26in} \ $ or the {\bf second type} $\ \diag{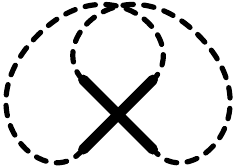}{0.34in}$. The shadows of the first type are locally checkerboard colorable and of the second type are not. The cubes of resolutions for these shadows are $\bullet \to\bullet$, and they have just one edge, which is a split/join for the shadow of the first type and a single cycle bifurcation for the shadow of the second type.

Now assume the lemma has been proved for all connected shadows with fewer than $n$ crossings. Let $D$ be a connected shadow with $n$ crossings. We will show that if $D$ is not locally checkerboard colorable, then there is a single cycle bifurcation in its cube of resolutions. Pick a crossing $x$ and let $D'$ be the diagram obtained by smoothing $D$ at $x$. (It does not matter which smoothing is chosen.)

Assume first that $D'$ is not locally checkerboard colorable. By induction, the cube of resolutions for $D'$ contains a single cycle bifurcation. Since the cube of resolutions of $D'$ is a face of the cube of resolutions of $D,$ the result follows.

On the other hand, if $D'$ is locally checkerboard colorable, then by \Cref{alt-check}, it admits an alternating state $S'$. We color $N_{D'}\sm D'$ consistently, so that the smoothings of $S'$ joins white regions. Let $S$ be a state of $D$ which coincides with $S'$, and  $S^\vee$ its dual state. Switching the smoothing of $x$ in $S^\vee$, we obtain $S'^\vee$,  considered as a state of $D$.

The ribbon surface $N_{D}$ is obtained by adding a 2-dimensional 1-handle (a band) to $N_{D'}$. Unless the transition from $S'^\vee$ to $S^\vee$ is a single cycle bifurcation, we can extend the coloring of $(N_{D'}, D')$ to $(N_{D}, D)$. Since $D$ is not locally checkerboard colorable, the transition from $S'^\vee$ to $S^\vee$ must be a single cycle bifurcation.
\end{proof}

Recall that $r(D)$ denotes the rank of the image of $i_* \colon H_1(D; \Z/2) \to H_1(\Sigma; \Z/2).$ Any connected shadow is homotopy equivalent to a bouquet of circles. If $D$ has $c(D)$ crossings, then $\chi(D)=-c(D)$. It follows that $0 \leq r(D) \leq c(D) + 1$ for connected shadows with $c(D)$ crossings.

\begin{proposition}\label{p-non-alt-ineq}
Let $D$ be a link shadow in $\Sigma$, (not necessarily connected).
\begin{itemize}
\item[(i)] If $S$ is a state of $D$, then  $$t(S)+t(S^\vee)\leq c(D)+2|D|-r(D). $$
\item[(ii)] If $D$ is not locally checkerboard colorable, then for any state $S$ of $D$, $$t(S)+t(S^\vee)< c(D)+2|D|-r(D).$$
\item[(iii)] If $D$ is strongly prime and $S$ is non-alternating, then $$t(S)+t(S^\vee)< c(D)+2|D|-r(D).$$
\end{itemize}
\end{proposition}

\begin{proof}
Let us write $\Sigma = \Sigma_1 \cup \dots \cup \Sigma_n$ as a disjoint union of connected components. Any component disjoint from $D$ does not contribute to the terms in (i), (ii) and (iii), so it can be discarded. Therefore, we can assume that $D_i= D \cap \Sigma_i \neq \varnothing$ for $i=1, \dots, n.$

Since all terms of the inequalities of the statements are additive under taking disjoint unions of surfaces, it is enough to prove the statement for $\Sigma$ connected.

On the other hand, if $D=D_1 \cup D_2$ is disconnected, then $r(D)\leq r(D_1)+r(D_2)$. Thus $r(D)$ is subadditive, and since the other terms on the right hand side of (i), (ii) and (iii) are additive, it is enough to prove the proposition for connected shadows in connected surfaces. Assume henceforth that $\Sigma$ is a connected surface.

Let us prove the statement for single crossing abstract shadows $D$ now. Recall from the proof of Lemma \ref{cc-single} that single crossing shadows $D$ are of two types. For both of them, $r(D) \leq 2.$ If $r(D)=0$, then $t(S)+t(S^\vee)= 2$. If $r(D)=1,2$, then $t(S)+t(S^\vee)\leq 1$. Therefore, statement (i) holds for 1-crossing shadows.  Since shadows of the first type are locally checkerboard colorable and $t(S)=t(S^\vee)=0$ for shadows of the second type, statements (ii) and (iii) hold as well.

The proof of (i) proceeds by induction on the crossing number $c(D)$. Let $D$ be a connected shadow in $\Sigma$ with $c(D) \geq 2$  crossings. We assume that statement (i) has been established for all connected shadows in $\Sigma$ with fewer than $c(D)$ crossings.

Let $D'$ be the shadow resulting from smoothing at a crossing $x$ of $D$. We choose the smoothing so that $D'$ is connected. Notice that
\begin{equation}\label{r-eqn}
r(D)-1 \leq r(D') \leq r(D).
\end{equation}

Let $S$ be a state of $D$. The chosen smoothing of $x$ coincides either with the smoothing of $x$ in $S$ or in $S^\vee$ and, without loss of generality, we can assume that it coincides with the smoothing of $x$ in $S$. Then $S$ induces a state on $D'$ denoted $S'$. Clearly, $t(S')=t(S).$ The dual state $S'^\vee$ to $S'$  differs from $S^\vee$ at $x$ only. The states $S^\vee$ and $S'^\vee$ are adjacent in the cube of resolutions of $D$. Thus
\begin{equation}\label{t-eqn}
t(S'^\vee)-1\leq t(S^\vee) \leq t(S'^\vee)+1.
\end{equation}

\begin{lemma}\label{in-line-lemma}
Either $r(D')= r(D)$ or $t(S^\vee) \leq t(S'^\vee)$.
\end{lemma}

\begin{proof}
Assume that $t(S^\vee) > t(S'^\vee)$.
Then either two trivial loops in $S^\vee$ join to make a trivial loop in $S'^\vee$, or a trivial and a nontrivial loop in $S^\vee$ join to make a nontrivial loop in $S'^\vee$, or a trivial loop in $S^\vee$ splits to make two nontrivial loops in $S'^\vee$. In each case, $r(D') = r(D).$
\end{proof}

We prove the inductive step for part (i). By Lemma \ref{in-line-lemma}, there are the two possibilities.
If $r(D')=r(D)$, then \Cref{t-eqn} and the inductive assumption imply that
$$t(S)+t(S^\vee)\leq t(S')+t(S'^\vee)+1 \leq c(D')+2-r(D')+1=c(D)+2-r(D).$$

On the other hand, if $r(D')\ne r(D)$, then $t(S^\vee) \leq t(S'^\vee)$, and \Cref{r-eqn} and the inductive assumption imply that
$$t(S)+t(S^\vee)\leq t(S')+t(S'^\vee) \leq c(D')+2-r(D')=c(D)+2-r(D).$$
This completes the proof in case (i).

We prove part (ii) also by induction on $c(D)$. Let $D$ be a connected shadow in $\Sigma$ with $c(D) \geq 2$  crossings, and assume $D$ is not locally checkerboard colorable. We assume that statement (ii)  has been established for all connected shadows in $\Sigma$ with fewer than $c(D)$ crossings that are not locally checkerboard colorable. By Lemma \ref{cc-single}, there is a single cycle bifurcation in the cube of resolutions of $D$.

Let $D'$ be the shadow resulting from smoothing $D$ at a crossing $x$, and we assume $D'$ is connected and that the smoothing at $x$ coincides with the smoothing of $x$ in $S$.

If $D'$ is locally checkerboard colorable, then the transition from $S^\vee$ to $S'^\vee$ must be a single cycle bifurcation, for otherwise the local checkerboard coloring would extend from $D'$  to $D$.

Since the transition is a single cycle bifurcation, we have $t(S^\vee)=t(S'^\vee)$ and $r(D)=r(D')$. Therefore, applying  part (i) to $D'$, we see that
$$t(S)+t(S^\vee) = t(S)+t(S'^\vee) \leq c(D')+2-r(D') < c(D)+2-r(D).$$

If $D'$ is not locally checkerboard colorable, then we can apply the inductive hypothesis for part (ii) to $D'$ and use it to deduce the desired strict inequality just as before.  This completes the proof of (ii).

The last step is to prove statement (iii). We begin by verifying (iii) for connected shadows with one or two crossings. For a single crossing shadow $D$ of the first type, both states are alternating, so (iii) is vacuously true.  Single crossing shadow  of the second type are not locally checkerboard colorable, and so the result follows from (ii).

\begin{figure}[!ht]
\centering
\includegraphics[height=35mm]{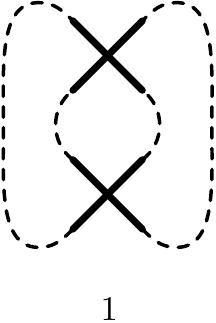}
\hspace{4mm}
\includegraphics[height=35mm]{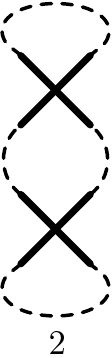}
\hspace{4mm}
\includegraphics[height=35mm]{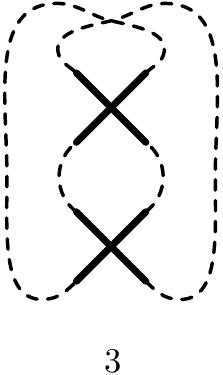}
\hspace{4mm}
\includegraphics[height=35mm]{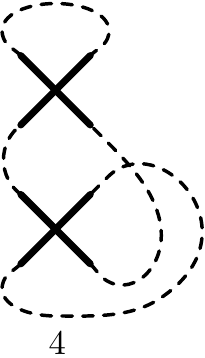}
\hspace{4mm}
\includegraphics[height=35mm]{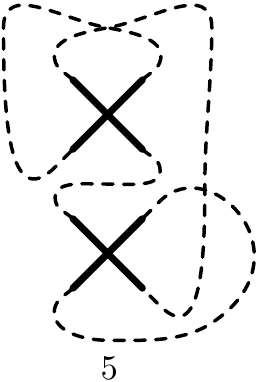}
\caption{Connected shadow diagrams with $2$-crossings.} \label{f-shadows}
\end{figure}

All abstract connected $2$-crossing shadows $D$ are depicted in \Cref{f-shadows}. For a type 1 shadow $D$, its non-alternating states appear in \Cref{f-2122} (left). Note that $0\leq r(D) \leq 3$ and $0\leq t(S),t(S^\vee)\leq 1.$ If $r(D)= 0$ or $1$, then  $t(S)+t(S^\vee) \leq 2$ and $3 \leq c(D)+2-r(D)$. Thus  (iii) holds in this case. If $r(D)=2$ or $3$, then $t(S)=t(S^\vee)=0$, and statement (iii) holds.

For a type 2 shadow $D$, its non-alternating states are shown in \Cref{f-2122} (right). Note that $0\leq r(D) \leq 3$ and $0\leq t(S),t(S^\vee)\leq 2.$  Since $D$ is strongly prime, $r(D)>0$ and $t(S),t(S^\vee)\leq 1.$ If $r(D)=1$, then $t(S)+t(S^\vee) \leq 2$; if $r(D)=2$, then $t(S)+t(S^\vee)\leq 1$; and  if $r(D)=3$, then $t(S)+t(S^\vee)=0$. In all three cases, statement (iii) is seen to hold.

\begin{figure}[!h]
\centering
\begin{subfigure}{0.36\textwidth}
\centering
\includegraphics[height=25mm]{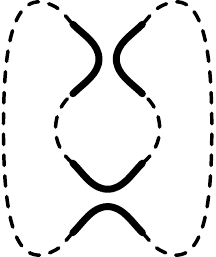}\hspace{.3cm}\includegraphics[height=25mm]{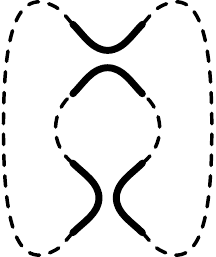}
\subcaption{Type 1.}
\end{subfigure}
\begin{subfigure}{0.36\textwidth}
\centering
\includegraphics[height=32mm]{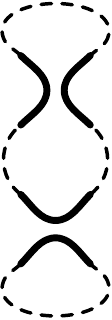}\hspace{.4cm}\includegraphics[height=32mm]{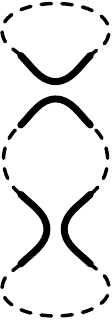}.
\subcaption{Type 2.}
\end{subfigure}
\caption{Non-alternating states on 2-crossing shadows of type 1 and 2.}\label{f-2122}
\end{figure}

Note that none of the shadows of the third, fourth, and fifth type are locally checkerboard colorable. Therefore statement (iii) follows from (ii) in these cases.

The proof of (iii) proceeds by induction on the crossing number $c(D)$. Let $D$ be a strongly prime connected shadow in $\Sigma$. By (ii), we can assume that $D$ is locally checkerboard colorable. We assume additionally that $c(D) \geq 3$ and that statement (iii) has been established for all strongly prime  shadows in $\Sigma$ with fewer than $c(D)$ crossings.

Let $S$ be a non-alternating state for $D$. Then $S$ has two consecutive smoothings that are identical, and we choose a third crossing $x$ of $D$. By \Cref{l-strongly-prime}, one of the smoothings of $x$ yields a shadow which is connected and strongly prime. Let $D'$ be the resulting shadow. As before, we assume that the smoothing at $x$ coincides with the smoothing of $x$ in $S$.
The state $S$ induces a state on $D'$, denoted $S'$, which is  non-alternating.
Since $D'$ is connected, one can apply Lemma \ref{in-line-lemma} as before and argue again by induction that (iii) holds for $D$.
\end{proof}

\subsection{Proof of \Cref{t-altern-span}}
\label{s-proof-altern-span}

Part (i) follows immediately by combining  \Cref{cor-adequate} and \Cref{p-non-alt-ineq} (i).

For parts (ii) and (iii), note that if $D$ is a connected sum of $D_0\subset \Sigma,$ and $D_1, \ldots, D_k\subset S^2$ then
\begin{equation} \label{e-csf}
[D]_\Sigma = \delta^{-k} [D_0]_\Sigma\cdot \prod_{i=1}^k [D_i]_{S^2}.
\end{equation}
Therefore, it is enough to prove parts (ii) and (iii) for \emph{prime} diagrams
(alternating for (ii) and non-alternating for (iii)).

The condition that $D$ is prime implies that it is not a nontrivial connected sum diagram as above. More precisely, a link diagram $D$ on $\Sigma$ is said to be  {\bf prime} if any contractible simple loop $\gamma$ in $\Sigma$ that meets $D$ transversely at two points bounds a 2-disk that intersects $D$ in an unknotted arc (possibly with self-crossings).

Proof of (iii):  Assume $D$ is prime. If the shadow diagram of $D$ is strongly prime, then the statement follows from \Cref{cor-adequate} and \Cref{p-non-alt-ineq} (iii). If it is not strongly prime then $D$ must contain a self-crossing trivial arc. Let $D'$ be obtained by replacing it by a simple trivial arc. Since ${\rm span}([D]_\Sigma)$ is invariant under Reidemeister moves and $r(D')=r(D)$,
$${\rm span}([D]_\Sigma)={\rm span}([D']_\Sigma)\leq 4c(D')+4|D'|-2r(D')< 4c(D)+4|D|-2r(D),$$
by part (i).

Our proof of (ii) follows that of \cite[Theorem 2.9]{Boden-Karimi-2019}.
Since both sides of the equality in (ii) are additive under disjoint union of diagrams, it is enough to prove it for connected diagrams.

By \Cref{p-alter-cc}, $D$ is checkerboard colorable. Then all regions of one color, say white, are enclosed by the cycles in the state $S_A$ of $D$, and all regions of the other color, i.e., black, are enclosed by the cycles in the state $S_B$. Therefore, the numbers of white and black regions are $t(S_A)$ and $t(S_B)$, respectively. Since $D$ defines a cellular decomposition of $\Sigma$ into $c(D)$ 0-cells, $2c(D)$ $1$-cells, and $t(S_A)+t(S_B)$ $2$-cells,
$$2-2g(\Sigma)=\chi(\Sigma)=c(D)-2c(D)+t(S_A)+t(S_B),$$
and $$t(S_A)+t(S_B)=c(D)+2-2g(\Sigma).$$ By \Cref{p-dminmax},
\begin{align*}
{\rm span}([D]_\Sigma)&= d_{max}([D]_\Sigma))-d_{min}([D]_\Sigma)),\\
&= 2c(D) + 2t(S_A)+2t(S_B),\\
&= 4c(D)+4-4g(\Sigma). \qedhere
\end{align*}

\bibliographystyle{alpha}

\begin{thebibliography}{AARH{\etalchar{+}}19}

\bibitem[AAR{\etalchar{+}}19]{Adams-2019a}
Colin Adams, Carlos Albors-Riera, Beatrix Haddock, Zhiqi Li, Daishiro Nishida,
  Braeden Reinoso, and Luya Wang.
\newblock Hyperbolicity of links in thickened surfaces.
\newblock {\em Topology Appl.}, 256:262--278, 2019.

\bibitem[AEG{\etalchar{+}}19]{Adams-2019c}
Colin Adams, Or~Eisenberg, Jonah Greenberg, Kabir Kapoor, Zhen Liang, Kate
  O'Connor, Natalia Pacheco-Tallaj, and Yi~Wang.
\newblock Turaev hyperbolicity of classical and virtual links, 2019.
\newblock \href{https://arxiv.org/pdf/1912.09435.pdf}{ArXiv/1912.09435}, to
  appear in Alg. Geom. Topol.

\bibitem[AFLT02]{Adams}
Colin Adams, Thomas Fleming, Michael Levin, and Ari~M. Turner.
\newblock Crossing number of alternating knots in {$S\times I$}.
\newblock {\em Pacific J. Math.}, 203(1):1--22, 2002.

\bibitem[BaK20]{Bavier-Kalfagianni-2020}
Brandon Bavier and Efstratia Kalfagianni.
\newblock Guts, volume and skein modules of 3-manifolds, 2020.
\newblock \href{https://arxiv.org/pdf/2010.06559.pdf}{ArXiv/2010.06559}.

\bibitem[BAR19]{Blair-Allen-Rodriguez-2019}
Ryan Blair, Heidi Allen, and Leslie Rodriguez.
\newblock Twist number and the alternating volume of knots.
\newblock {\em J. Knot Theory Ramifications}, 28(1):1950016, 16, 2019.

\bibitem[BFK99]{BFK-1999}
Doug Bullock, Charles Frohman, and Joanna Kania-Bartoszy\'{n}ska.
\newblock Understanding the {K}auffman bracket skein module.
\newblock {\em J. Knot Theory Ramifications}, 8(3):265--277, 1999.

\bibitem[BHMV95]{BHMV-1995}
Christian Blanchet, Nathan Habegger, Gregor Masbaum, and Pierre Vogel.
\newblock Topological quantum field theories derived from the {K}auffman
  bracket.
\newblock {\em Topology}, 34(4):883--927, 1995.

\bibitem[BK19]{Boden-Karimi-2019}
Hans~U. Boden and Homayun Karimi.
\newblock The {J}ones-{K}rushkal polynomial and minimal diagrams of surface links, 2019.
\newblock \href{https://arxiv.org/pdf/1908.06453}{ArXiv/1908.06453},
to appear in Ann. Inst. Fourier (Grenoble).

\bibitem[BK20]{BK-2020}
Hans~U. Boden and Homayun Karimi.
\newblock A characterization of alternating links in thickened surfaces, 2020.
\newblock \href{https://arxiv.org/pdf/2010.14030}{ArXiv/2010.14030},
to appear in Proc. Roy. Soc. Edinburgh Sect. A

\bibitem[Bla09]{Blair-2009}
Ryan~C. Blair.
\newblock Alternating augmentations of links.
\newblock {\em J. Knot Theory Ramifications}, 18(1):67--73, 2009.

\bibitem[BR21]{BR-2021}
Hans~U. Boden and William Rushworth.
\newblock Minimal crossing number implies minimal supporting genus.
\newblock{\em Bull. Lond. Math. Soc.}, 53(4):1174--1184, 2021.

\bibitem[Bul97]{Bullock-1997}
Doug Bullock.
\newblock Rings of {${\rm SL}_2({\bf C})$}-characters and the {K}auffman
  bracket skein module.
\newblock {\em Comment. Math. Helv.}, 72(4):521--542, 1997.

\bibitem[BW11]{BW-2011}
Francis Bonahon and Helen Wong.
\newblock Quantum traces for representations of surface groups in {${\rm
  SL}_2(\Bbb C)$}.
\newblock {\em Geom. Topol.}, 15(3):1569--1615, 2011.

\bibitem[CK14]{CK-Turaev}
Abhijit Champanerkar and Ilya Kofman.
\newblock A survey on the {T}uraev genus of knots.
\newblock {\em Acta Math. Vietnam.}, 39(4):497--514, 2014.

\bibitem[CK20]{Champanerkar-Kofman-2020}
Abhijit Champanerkar and Ilya Kofman.
\newblock A volumish theorem for alternating virtual links, 2020.
\newblock \href{https://arxiv.org/pdf/2010.08499.pdf}{ArXiv/2010.08499}.

\bibitem[CKS02]{Carter-Kamada-Saito}
J.~Scott Carter, Seiichi Kamada, and Masahico Saito.
\newblock Stable equivalence of knots on surfaces and virtual knot cobordisms.
\newblock {\em J. Knot Theory Ramifications}, 11(3):311--322, 2002.
\newblock Knots 2000 Korea, Vol. 1 (Yongpyong).

\bibitem[CL19]{CL-2019}
Francesco Costantino and Thang T.~Q. L{\^e}.
\newblock Stated skein algebras of surfaces, 2019.
\newblock \href{https://arxiv.org/pdf/1907.11400.pdf}{ArXiv/1907.11400}.

\bibitem[DFK{\etalchar{+}}08]{DFKLS}
Oliver~T. Dasbach, David Futer, Efstratia Kalfagianni, Xiao-Song Lin, and
  Neal~W. Stoltzfus.
\newblock The {J}ones polynomial and graphs on surfaces.
\newblock {\em J. Combin. Theory Ser. B}, 98(2):384--399, 2008.

\bibitem[Dia04]{Diao-2004}
Yuanan Diao.
\newblock The additivity of crossing numbers.
\newblock {\em J. Knot Theory Ramifications}, 13(7):857--866, 2004.

\bibitem[DK05]{Dye-Kauffman-2005}
Heather~A. Dye and Louis~H. Kauffman.
\newblock Minimal surface representations of virtual knots and links.
\newblock {\em Algebr. Geom. Topol.}, 5:509--535, 2005.

\bibitem[DL07]{Dasbach-Lin-2007}
Oliver~T. Dasbach and Xiao-Song Lin.
\newblock A volumish theorem for the {J}ones polynomial of alternating knots.
\newblock {\em Pacific J. Math.}, 231(2):279--291, 2007.

\bibitem[FG06]{FGo-2006}
Vladimir Fock and Alexander Goncharov.
\newblock Moduli spaces of local systems and higher {T}eichm\"{u}ller theory.
\newblock {\em Publ. Math. Inst. Hautes \'{E}tudes Sci.}, (103):1--211, 2006.

\bibitem[FGL02]{FGL-2002}
Charles Frohman, Razvan Gelca, and Walter Lofaro.
\newblock The {A}-polynomial from the noncommutative viewpoint.
\newblock {\em Trans. Amer. Math. Soc.}, 354(2):735--747, 2002.

\bibitem[FKL19]{FKL-2019}
Charles Frohman, Joanna Kania-Bartoszynska, and Thang T.~Q. L{\^e}.
\newblock Unicity for representations of the {K}auffman bracket skein algebra.
\newblock{\em Invent. math.} 215(2):609--650, 2019.

\bibitem[FST08]{FST-2008}
Sergey Fomin, Michael Shapiro, and Dylan Thurston.
\newblock Cluster algebras and triangulated surfaces. {I}. {C}luster complexes.
\newblock {\em Acta Math.}, 201(1):83--146, 2008.

\bibitem[How15]{Howie-2015}
Joshua Howie.
\newblock {\em Surface-alternating knots and links}.
\newblock PhD thesis, University of Melbourne, 2015.

\bibitem[HP20]{HP-2020}
Joshua~A. Howie and Jessica~S. Purcell.
\newblock Geometry of alternating links on surfaces.
\newblock {\em Trans. Amer. Math. Soc.}, 373(4):2349--2397, 2020.

\bibitem[Kam02]{Kamada-2002}
Naoko Kamada.
\newblock On the {J}ones polynomials of checkerboard colorable virtual links.
\newblock {\em Osaka J. Math.}, 39(2):325--333, 2002.

\bibitem[KK00]{Kamada-2000}
Kamada, Naoko and Kamada, Seiichi.
\newblock Abstract link diagrams and virtual knots.
\newblock {\em J. Knot Theory Ramifications}, 9(1):93--106, 2000.

\bibitem[Kar21]{Karimi-2019}
Homayun Karimi.
\newblock The {K}hovanov homology of alternating virtual links.
\newblock {\em Michigan Math. J.}, 70(4):749--778, 2021.


\bibitem[Kau87]{Kauffman-87}
Louis~H. Kauffman.
\newblock State models and the {J}ones polynomial.
\newblock {\em Topology}, 26(3):395--407, 1987.

\bibitem[Kru11]{Krushkal-2011}
Vyacheslav Krushkal.
\newblock Graphs, links, and duality on surfaces.
\newblock {\em Combin. Probab. Comput.}, 20(2):267--287, 2011.

\bibitem[Kup03]{Kuperberg}
Greg Kuperberg.
\newblock What is a virtual link?
\newblock {\em Algebr. Geom. Topol.}, 3:587--591 (electronic), 2003.

\bibitem[Lac04]{Lackenby-2004}
Marc Lackenby.
\newblock The volume of hyperbolic alternating link complements.
\newblock {\em Proc. London Math. Soc. (3)}, 88(1):204--224, 2004.
\newblock With an appendix by Ian Agol and Dylan Thurston.

\bibitem[Lac09]{Lackenby-2009}
Marc Lackenby.
\newblock The crossing number of composite knots.
\newblock {\em J. Topol.}, 2(4):747--768, 2009.

\bibitem[L{\^e}06]{Le-2006}
Thang T.~Q. L{\^e}.
\newblock The colored {J}ones polynomial and the {$A$}-polynomial of knots.
\newblock {\em Adv. Math.}, 207(2):782--804, 2006.

\bibitem[Lic97]{Lickorish}
W.~B.~Raymond Lickorish.
\newblock {\em An introduction to knot theory}, volume 175 of {\em Graduate
  Texts in Mathematics}.
\newblock Springer-Verlag, New York, 1997.

\bibitem[LT88]{Lickorish-Thistlethwaite}
W.~B.~Raymond Lickorish and Morwen~B. Thistlethwaite.
\newblock Some links with nontrivial polynomials and their crossing-numbers.
\newblock {\em Comment. Math. Helv.}, 63(4):527--539, 1988.

\bibitem[Man13]{Manturov-2013}
Vassily~O. Manturov.
\newblock Parity and projection from virtual knots to classical knots.
\newblock {\em J. Knot Theory Ramifications}, 22(9):1350044, 20, 2013.

\bibitem[Men84]{Menasco-1984}
William Menasco.
\newblock Closed incompressible surfaces in alternating knot and link
  complements.
\newblock {\em Topology}, 23(1):37--44, 1984.

\bibitem[MT93]{Menasco-Thistlethwaite-93}
William Menasco and Morwen Thistlethwaite.
\newblock The classification of alternating links.
\newblock {\em Ann. of Math. (2)}, 138(1):113--171, 1993.

\bibitem[Mul16]{Muller-2016}
Greg Muller.
\newblock Skein and cluster algebras of marked surfaces.
\newblock {\em Quantum Topol.}, 7(3):435--503, 2016.

\bibitem[Mur87]{Murasugi-871}
Kunio Murasugi.
\newblock Jones polynomials and classical conjectures in knot theory.
\newblock {\em Topology}, 26(2):187--194, 1987.

\bibitem[Prz99]{Przytycki-1999}
J\'{o}zef~H. Przytycki.
\newblock Fundamentals of {K}auffman bracket skein modules.
\newblock {\em Kobe J. Math.}, 16(1):45--66, 1999.

\bibitem[PS00]{PS-2000}
J\'{o}zef~H. Przytycki and Adam~S. Sikora.
\newblock On skein algebras and {${\rm Sl}_2({\bf C})$}-character varieties.
\newblock {\em Topology}, 39(1):115--148, 2000.

\bibitem[Sto94]{Stong-1994}
Richard Stong.
\newblock The {J}ones polynomial of parallels and applications to crossing
  number.
\newblock {\em Pacific J. Math.}, 164(2):383--395, 1994.

\bibitem[SW07]{Sikora-Westbury}
Adam~S. Sikora and Bruce~W. Westbury.
\newblock Confluence theory for graphs.
\newblock {\em Algebr. Geom. Topol.}, 7:439--478, 2007.

\bibitem[Thi87]{Thistlethwaite-87}
Morwen~B. Thistlethwaite.
\newblock A spanning tree expansion of the {J}ones polynomial.
\newblock {\em Topology}, 26(3):297--309, 1987.

\bibitem[Tur87]{Turaev-1987}
Vladimir~G. Turaev.
\newblock A simple proof of the {M}urasugi and {K}auffman theorems on
  alternating links.
\newblock {\em Enseign. Math. (2)}, 33(3-4):203--225, 1987.

\bibitem[Tur88]{Turaev-1990}
Vladimir~G. Turaev.
\newblock The {C}onway and {K}auffman modules of a solid torus.
\newblock {\em Zap. Nauchn. Sem. Leningrad. Otdel. Mat. Inst. Steklov. (LOMI)},
  167(Issled. Topol. 6):79--89, 190, 1988.

\bibitem[Tur91]{Turaev-1991}
Vladimir~G. Turaev.
\newblock Skein quantization of {P}oisson algebras of loops on surfaces.
\newblock {\em Ann. Sci. \'{E}cole Norm. Sup. (4)}, 24(6):635--704, 1991.

\bibitem[Tur94]{Turaev-1994}
Vladimir~G. Turaev.
\newblock Algebras of loops on surfaces, algebras of knots, and quantization.
\newblock In {\em Braid group, knot theory and statistical mechanics, {II}},
  volume~17 of {\em Adv. Ser. Math. Phys.}, pages 324--360. World Sci. Publ.,
  River Edge, NJ, 1994.

\bibitem[Wil20]{Will-2020}
David~A. Will.
\newblock Homological polynomial coefficients and the twist number of
  alternating surface links, 2020.
\newblock \href{https://arxiv.org/pdf/2011.12274.pdf}{ArXiv/2011.12274},
to appear in Alg. Geom. Topol.

\end{thebibliography}

\newcommand{\etalchar}[1]{$^{#1}$}

\end{document}